\documentclass[12pt, paper=a4, fleqn]{scrartcl}
\usepackage[utf8]{inputenc}

\usepackage[utf8]{inputenc}

\usepackage{amsmath}
\usepackage{amssymb}
\usepackage{mathtools}

\mathtoolsset{centercolon}

\DeclarePairedDelimiter\paren{\lparen}{\rparen}

\DeclarePairedDelimiter\brackets{[}{]}
\DeclarePairedDelimiter\abs{\lvert}{\rvert}
\DeclarePairedDelimiter\norm{\lVert}{\rVert}

\DeclarePairedDelimiterX{\closedStochasticInterval}[1]{[}{]}{\!\delimsize[#1\delimsize]\!}
\DeclarePairedDelimiterX{\leftOpenStochasticInterval}[1]{]}{]}{\!\delimsize]#1\delimsize]\!}
\DeclarePairedDelimiterX{\rightOpenStochasticInterval}[1]{[}{[}{\!\delimsize[#1\delimsize[\!}
\DeclarePairedDelimiterX{\openStochasticInterval}[1]{]}{[}{\!\delimsize]#1\delimsize[\!}

\newcommand{\bp}[1]{\paren[\big]{#1}}

\newcommand{\diff}{\,\mathrm d}

\newcommand{\EE}{\mathds{E}}
\newcommand{\PP}{\mathds{P}}

\newcommand{\RR}{\mathds{R}}

\newcommand{\scA}{\mathcal{A}}
\newcommand{\scB}{\mathcal{B}}
\newcommand{\scC}{\mathcal{C}}
\newcommand{\scE}{\mathcal{E}}
\newcommand{\scF}{\mathcal{F}}

\newcommand{\scS}{\mathcal{S}}
\newcommand{\scW}{\mathcal{W}}

\newcommand{\overbar}[1]{\mkern 1.5mu\overline{\mkern-1.5mu#1\mkern-1.5mu}\mkern 1.5mu}

\newcommand{\squeeze}[2][0]{%
  \mbox{$\medmuskip=#1mu\displaystyle#2$}%
}

\newcommand{\comment}[1]{}

\newcommand{\baseS}{\widebar{S}}
\newcommand{\admissibleSellStrategies}[1]{\scA_\text{mon}(#1)}
\newcommand{\admissibleFiniteVariationStrategies}[1]{\scA_\text{bv}(#1)}

\newcommand{\buyReg}{\scB}

\let\originaltheta\theta

\newcommand{\assetsProcess}{\Theta}

\newcommand{\lagrangeMult}{\widebar \lambda}
\newcommand{\ybar}{\widebar y}
\newcommand{\thetabar}{\widebar \originaltheta}

\newcommand{\Vboundary}{V_\text{bdry}}
\newcommand{\Vbuysell}{V^{\buyReg,\scS}}

\makeatletter
\newcommand*\if@single[3]{%
  \setbox0\hbox{${\mathaccent"0362{#1}}^H$}%
  \setbox2\hbox{${\mathaccent"0362{\kern0pt#1}}^H$}%
  \ifdim\ht0=\ht2 #3\else #2\fi
  }
\newcommand*\rel@kern[1]{\kern#1\dimexpr\macc@kerna}
\newcommand*\widebar[1]{\@ifnextchar^{{\wide@bar{#1}{0}}}{\wide@bar{#1}{1}}}
\newcommand*\wide@bar[2]{\if@single{#1}{\wide@bar@{#1}{#2}{1}}{\wide@bar@{#1}{#2}{2}}}
\newcommand*\wide@bar@[3]{%
  \begingroup
  \def\mathaccent##1##2{%
    \if#32 \let\macc@nucleus\first@char \fi
    \setbox\z@\hbox{$\macc@style{\macc@nucleus}_{}$}%
    \setbox\tw@\hbox{$\macc@style{\macc@nucleus}{}_{}$}%
    \dimen@\wd\tw@
    \advance\dimen@-\wd\z@
    \divide\dimen@ 3
    \@tempdima\wd\tw@
    \advance\@tempdima-\scriptspace
    \divide\@tempdima 10
    \advance\dimen@-\@tempdima
    \ifdim\dimen@>\z@ \dimen@0pt\fi
    \rel@kern{0.6}\kern-\dimen@
    \if#31
      \overline{\rel@kern{-0.6}\kern\dimen@\macc@nucleus\rel@kern{0.4}\kern\dimen@}%
      \advance\dimen@0.4\dimexpr\macc@kerna
      \let\final@kern#2%
      \ifdim\dimen@<\z@ \let\final@kern1\fi
      \if\final@kern1 \kern-\dimen@\fi
    \else
      \overline{\rel@kern{-0.6}\kern\dimen@#1}%
    \fi
  }%
  \macc@depth\@ne
  \let\math@bgroup\@empty \let\math@egroup\macc@set@skewchar
  \mathsurround\z@ \frozen@everymath{\mathgroup\macc@group\relax}%
  \macc@set@skewchar\relax
  \let\mathaccentV\macc@nested@a
  \if#31
    \macc@nested@a\relax111{#1}%
  \else
    \def\gobble@till@marker##1\endmarker{}%
    \futurelet\first@char\gobble@till@marker#1\endmarker
    \ifcat\noexpand\first@char A\else
      \def\first@char{}%
    \fi
    \macc@nested@a\relax111{\first@char}%
  \fi
  \endgroup
}
\makeatother

\usepackage{amsthm}
\usepackage[english]{babel}
\usepackage{calc}
\usepackage{caption}
\usepackage{subcaption}
\usepackage{color}
\usepackage{csquotes}
\usepackage{dsfont}
\usepackage{graphicx}
\usepackage{import}
\usepackage{microtype}
\usepackage{nameref}
\usepackage{overpic}
\usepackage[section]{placeins}
\usepackage{tikz}
\usepackage{units}

\usepackage{xr}%
\usepackage[unicode,pdfborder={0 0 0}]{hyperref}
\usepackage[noabbrev,capitalize]{cleveref}

\usetikzlibrary{calc}

\newtheoremstyle{boldremark}
	{\topsep}   
	{\topsep}   
	{}          
	{}          
	{\bfseries} 
	{.}         
	{.5em}      
	{}          

\newtheorem{theorem}{Theorem}[section]
\newtheorem{proposition}[theorem]{Proposition} 
\newtheorem{lemma}[theorem]{Lemma}

\theoremstyle{definition}
 
\newtheorem{assumption}[theorem]{Assumption}

\theoremstyle{boldremark}
\newtheorem{remark}[theorem]{Remark}
\newtheorem{example}[theorem]{Example}

\AtBeginDocument{%
	\crefname{equation}{equation}{equations}%
}

\numberwithin{equation}{section}

\author{Dirk Becherer\footnote{We thank Peter Bank for fruitful discussions on an early version of the control problem.}
 , Todor Bilarev\footnote{Support by German Science foundation DFG via Berlin Mathematical School BMS and 
research training group RTG1845 StoA  is gratefully acknowledged.} , Peter Frentrup\footnote{Email addresses: becherer,bilarev,frentrup@math.hu-berlin.de}\\
Institute of Mathematics, Humboldt-Universität zu Berlin
}

\title{Optimal Asset Liquidation with Multiplicative Transient Price Impact}

\begin{document}

\maketitle

\begin{abstract}
We study a multiplicative transient price impact model for an illiquid financial market, where trading causes price impact which is multiplicative in relation to the current price, transient over time with finite rate of resilience, and non-linear in the order size.
We construct explicit solutions  for the optimal control and the value function of singular optimal control problems to maximize expected discounted proceeds from liquidating a given asset position.
A free boundary problem, describing the optimal control, is solved for two variants of the problem where admissible controls are monotone or of bounded variation.

\vspace*{2ex}

 {\textbf{MSC2010 subject classifications}:  35R35, 49J40, 49L20, 60H30, 93E20, 91G80}

\end{abstract}

\section{Introduction}

We consider the optimal execution problem for a large trader in an illiquid financial market, 
who aims to sell (or buy, cf.\ Remark~\ref{problem:optimalaquisition}) a given amount of a risky asset, and derive explicit solutions for the optimal control and the related free boundary.
Since orders of the large trader have an adverse impact on the prices at which they are executed, she needs to balance the incurred liquidity costs against her preference to complete a trade early.
Optimal trade execution problems have been studied by many authors. 
We mention \cite{%
BertsimasLo98,%
ObizhaevaWang13,%
AlfonsiFruthSchied10,%
KharroubiPham10,%
AlfonsiSchiedSlynko12,%
Forsyth12,%
LorenzSchied13%
} and refer to \cite{PredoiuShaikhetShreve11,GatheralSchied13} for further references and application background.
  Posing the problem in continuous time leads to a singular stochastic control problem of finite fuel type.  
We note that our control objective, see \eqref{def:liquidation proceeds}--\eqref{task:maximize expected liquidation proceeds}, involves control cost terms like in \cite{Taksar97,DavisZervos98,DufourMiller04}, depending explicitly on the state process $(\baseS,Y)$ with 
a summation of integrals for each jump in the control strategy $\assetsProcess$.
We refer to these articles for more background on singular stochastic control.  
The articles \cite{Taksar97,DufourMiller04} show general results on existence for optimal singular controls; explicit descriptions of those can be obtained only for special problems, see e.g.\ \cite{KaratzasShreve86,Kobila93,DavisZervos98}, but these examples differ from the one considered here in several aspects.

In this paper we investigate a multiplicative limit order book model, which is closely related to the additive limit order book models of \cite{PredoiuShaikhetShreve11,AlfonsiFruthSchied10,ObizhaevaWang13,LorenzSchied13},
a key difference being that the price impact of orders is multiplicative instead of additive.
In absence of large trader activity, the risky asset price follows some unaffected
non-negative price evolution $\baseS=(\baseS_t)$, for instance
geometric Brownian motion.
 The trading strategy $(\assetsProcess_t)$ of the large trader
 has a multiplicative impact on the actual
asset price which is evolving as $S_t=\baseS_t f(Y_t)$, $t\ge 0$, for a process $Y$ that describes the level of market impact.
This process is defined by a mean-reverting differential equation
$\diff Y_t = -h(Y_t)\diff t + \diff\assetsProcess_t$, which is driven by  the amount $\assetsProcess_t$ of risky assets held,  and 
can be interpreted as a volume effect process like in \cite{PredoiuShaikhetShreve11,AlfonsiFruthSchied10}, see \cref{sect:LOB}.
Subject to suitable properties for the functions $f,h$ (see Assumption~\ref{cond:model parameters}), 
 asset sales (buys) are  depressing (increasing) the level of market impact $Y_t$ 
and thereby the actual price $S_t$ in a transient way, with some finite rate of resilience.
For $f$ being  positive, multiplicative price impact ensures
 that risky asset prices $S_t$ are  positive, like in the continuous-time variant \cite[Sect.~3.2]{GatheralSchied13}  of the
 model in \cite{BertsimasLo98}, whereas negative prices can occur in additive impact models.
 We admit for general non-linear impact functions $f$, corresponding to general density shapes of a multiplicative limit order book 
 whose shapes are specified with respect to relative price perturbations $S/{\baseS}$, 
and depth of the order book could be infinite or finite, cf.\ Sect.~\ref{sect:LOB}.
The rate of resilience $h(Y_t)/Y_t$ may be non-constant and (unaffected) transient recovery of $Y_t$ could be non-exponential,
while the problem still remains Markovian in $(\baseS,Y)$ through $Y$, like in \cite{PredoiuShaikhetShreve11} but differently to \cite{AlfonsiFruthSchied10,LorenzSchied13}. 
Following \cite{PredoiuShaikhetShreve11,GuoZervos13}, we admit for general (monotone) bounded variation strategies 
in continuous time, while \cite{AlfonsiFruthSchied10, KharroubiPham10} consider trading at discrete times.

Most of the related literature, like \cite{AlfonsiFruthSchied10,PredoiuShaikhetShreve11}, on transient additive price impact assumes that the unaffected (discounted) price dynamics exhibit no drift, and such a martingale property allows for different arguments in the analysis. 
Without drift, a convexity argument as in \cite{PredoiuShaikhetShreve11} can be applied readily also for multiplicative impact to identify the optimal control in the finite horizon problem with a free boundary that is constant in one coordinate, see \cref{rmk: finite time horizon; no drift; convexity}.
\cite{Lokka12} has shown how a multiplicative limit order book (cf.\ \cref{sect:LOB}) could be transformed into an additive one with further intricate dependencies, to which the result by \cite{PredoiuShaikhetShreve11} may be applied.
For additive impact, \cite{LorenzSchied13}  investigate the problem with general drift for finite horizon, whereas we
derive explicit solutions for multiplicative impact, infinite horizon and negative drift. 
The interesting articles \cite{KharroubiPham10,Forsyth12,GuoZervos13} also solve optimal trade execution problems in a model with multiplicative instead of additive price impact, but models and results differ in key aspects. The article \cite{GuoZervos13} considers  permanent price impact, non-zero bid-ask spread (proportional transaction costs) and a particular exponential parametrization for price impact from block trades, whereas we study transient price impact,  general
impact functions $f$, and zero spread (in Section~\ref{sect:intermediate buy orders}).
Numerical solutions of the Hamilton-Jacobi-Bellman equation derived by heuristic arguments are investigated in \cite{Forsyth12} for a different optimal execution problem on finite horizon  in a Black-Scholes model with permanent multiplicative impact.
The authors of  \cite{KharroubiPham10} obtain viscosity solutions and their nonlinear transient price impact is a functional of the present order size and the time lag from (only) the last trade, whereas we consider impact which depends via $Y$ on the times and sizes of all past orders, as in  \cite{PredoiuShaikhetShreve11}.

We construct explicit solutions for the optimal control which maximizes the expected discounted liquidation proceeds over an infinite time horizon, in a model with multiplicative price impact and drift that is introduced in \cref{sect:model}.  
We use dynamical programming and apply  smooth pasting and calculus
of variations methods to construct in \cref{sec:Free boundary problem} a candidate solution for the variational inequalities arising from the control problem. 
After having the candidate
value function and free boundary curve that determines the optimal control,
we prove 
optimality by verifying the variational inequalities (in \cref{appendix:proofs about value functions}) such that an optimality  principle  (see \cref{prop: supermatringale suffices}) can be applied.  
We obtain explicit solutions for two variants of the optimal liquidation problem.
In the first variant (I), whose solution is presented in \cref{sec:Problem}, the large trader is only admitted to sell but not to buy, whereas for the second variant (II) in \cref{sect:intermediate buy orders}
 intermediate buying is admitted, even though the trader ultimately wants to liquidate her position.  
Variant I may be of interest, if a bank selling a large position on behalf of a client is required by regulation
to execute only sell orders. The second variant might fit for an investor trading for herself and is mathematically needed to explore, whether a multiplicative limit order book model admits profitable round trips or transaction triggered price manipulations, as  studied by \cite{AlfonsiSchiedSlynko12} for additive impact, see \cref{rem:round-trips-in-depressed-market,ex:lorenz-schied}.
Notably, the free boundaries coincide
for both variants, and the time to complete liquidation is finite, varies continuously with the discounting parameter (i.e.~the investor's impatience) and  tends to zero for increasing impatience 
in suitable parametrizations, see~\cref{ex:simple-lob,fig:ttl-over-y0-different-delta}. 
In \cref{ex:lorenz-schied} we compare how additive and multiplicative limit order books give rise to rather different qualitative properties of optimal controls under
 standard Black-Scholes dynamics for unaffected risky asset prices,
 indicating that multiplicative impact fits better to models with multiplicative evolution of asset prices.

\section{Transient and multiplicative price impact}\label{sect:model}

We consider a filtered probability space $(\Omega,\scF,(\scF_t)_{t\ge 0},\PP)$.
The filtration $(\scF_t)_{t\geq 0}$ is assumed to satisfy the usual conditions of right-continuity and completeness,
  all semimartingales have c\`{a}dl\`{a}g paths, and (in)equalities of random variables are meant to hold almost everywhere. We refer to \cite{JacodShiryaev2003_book} for terminology and notations from stochastic analysis. We take $\scF_0$ to be trivial
and let also $\scF_{0-}$ denote the trivial $\sigma$-field. 
We consider a market with a risky asset in addition to the riskless numeraire asset, whose (discounted) price is constant at $1$.
Without trading activity of a large trader, the unaffected (fundamental) price process $S$ of the risky asset
 would be of the form
\begin{equation}\label{eq:baseS def}
	\baseS_t = e^{\mu t} M_t, \qquad \baseS_0 \in (0,\infty),
\end{equation}
 with $\mu \in \RR$ and with $M$ being a non-negative martingale that   is  square integrable on any compact time interval,
i.e.\ $\sup_{t\le T} E[M_t^2]<\infty$ for all $T\in [0,\infty)$, and  quasi-left continuous (cf.\  \cite{JacodShiryaev2003_book}), i.e.\ $\Delta M_{\tau} :=M_{\tau}-M_{\tau-}=0$  for any finite predictable stopping time $\tau$. 
Let us assume that the unaffected market is free of arbitrage for small investors in the sense that $\baseS$ is a local $\mathbb{Q}$-martingale under some probability measure $\mathbb{Q}$ that is locally equivalent to $\mathbb{P}$, i.e.\ $\mathbb{Q}\sim \mathbb{P}$ on $\scF_T$ for any $T\in [0,\infty)$. This implies no free lunch with vanishing risk \cite{DelbaenSchachermayer98} on any finite horizon $T$ for small investors. 
 The prime example where our assumptions on $M$ are satisfied is the Black-Scholes-Merton model, where $M = \mathcal{E}(\sigma W)$ is the stochastic exponential of a  Brownian motion $W$ scaled by  $\sigma > 0$.
 More generally, $M=\mathcal{E}(L)$ could be the stochastic exponential of a local martingale $L$, which is  a L\'evy process  with $\Delta L> -1$ and $\mathbb{E}[M_1^2]<\infty$ and such that $\baseS$ is not monotone (see \cite[Lemma~4.2]{Kallsen00} and \cite[Theorem~9.9]{ContTankov_book}),
 or one could have $M=\mathcal{E}(\int\sigma_t dW_t)$ for predictable stochastic volatility process  $(\sigma_t)_{t\ge0}$ that is bounded  in $[1/c, c]$, for $c>1$.   

The large trader's strategy $(\assetsProcess_t)_{t\geq 0}$ is her position in the risky asset.
Herein, $\assetsProcess_{0-}\ge 0$ denotes the initial position, $\assetsProcess_{0-} - \assetsProcess_t$ is the cumulative number of risky assets sold until time $t$.
The process $\assetsProcess$ is predictable, càdlàg and non-negative, i.e.\ short sales are not permitted, like for instance in \cite{KharroubiPham10,GuoZervos13}. Disallowing short sales is sensible for the control problem 
 with infinite horizon and negative drift to ensure existence of optimizers and finite time to complete liquidation; It is also supported e.g.\ by \cite[Remark 3.1]{Schied13}.
At first we do moreover assume $\assetsProcess$ to be decreasing, but this will be generalized later in \cref{sect:intermediate buy orders} to 
non-monotone strategies of bounded variation.

The large trader is faced with illiquidity costs, since trading causes adverse impact on the prices at which orders are executed, as follows. 
A process $Y$, the \emph{market impact process}, captures the price impact from strategy $\assetsProcess$, and is defined as the 
solution to
\begin{equation}\label{eq:deterministic Y_t dynamics}
	\diff Y_t = -h(Y_t) \diff t + \diff \assetsProcess_t 
\end{equation}
for some given initial condition $Y_{0-} \in \RR$.
Let $h: \RR \rightarrow \RR$ be strictly increasing and continuous with $h(0) = 0$. Further conditions will be imposed later in \cref{cond:model parameters}.
The market is resilient in that market impact $Y$ tends back towards its neutral level $0$ over time when the large trader is not active. Resilience is transient
with resilience rate $h(Y_t)$ that could be non-linear and is specified by the \emph{resilience function} $h$.
For example, the market recovers at exponential rate $\beta > 0$  (as in \cite{ObizhaevaWang13}) when  $h(y) = \beta y$ is linear. 
Clearly, $Y$ depends on $\assetsProcess$ and occasionally we will emphasize this by writing  $Y = Y^\assetsProcess$.

The actual (quoted) risky asset price $S$ is affected by the strategy $\assetsProcess$ of the large trader in a multiplicative way through the market impact process $Y$, and is modeled by
\begin{equation}\label{def:fctf}
	S_t := f(Y_t) \baseS_t,
\end{equation}
for an increasing function $f$ of the form 
\begin{equation} \label{eq:f and lambda}
	f(y) = \exp\Big(\int_0^y \lambda(x) \diff x \Big),\quad y\in\RR,
\end{equation}
with $\lambda: \RR \to (0,\infty)$ satisfying \cref{cond:model parameters} below, in particular being locally integrable.
For strategies $\assetsProcess$ that are continuous, the process $(S_t)_{t\geq 0}$ can be seen as the evolution of prices at which the trading strategy $\assetsProcess$ is executed.
That means, if the large trader is selling risky assets according to a continuous strategy $\assetsProcess^c$, then respective (self-financing) variations of her numéraire (cash) account are given by the proceeds (negative costs) $-\int_0^T S_u \diff \assetsProcess^c_u$ over any period $[0,T]$.
To permit also for non-continuous trading involving block trades, the proceeds from a market sell order of size $\Delta \assetsProcess_t\in \mathbb{R}$ at time $t$, are given by the term
\begin{equation}
	-\baseS_{t}\int_{0}^{\Delta \assetsProcess_t} f(Y_{t-} + x) \diff x, \label{eq:block sale proceeds}
\end{equation}
which is explained from executing the block trade within a (shadow) limit order book, see \cref{sect:LOB}.
Mathematically, defining proceeds from block trades in this way ensures good stability %
properties for proceeds defined by \eqref{def:liquidation proceeds}
as a function %
of strategies $\assetsProcess$, cf.\ \cite{BechererBilarevFrentrup-model-properties}.
In particular, approximating a block trade by a sequence of continuous trades executed over a shorter and shorter time interval yields the term \eqref{eq:block sale proceeds} in the limit.

\subsection{Limit order book perspective for multiplicative market impact}
\label{sect:LOB}

Multiplicative price impact and the proceeds from block trading can be interpreted by trading in a shadow limit order book (LOB).
We now show how the multiplicative price impact function $f$ is related to a LOB shape that is specified in terms of {\em relative} price pertubations $\rho_t:=S_t/\baseS_t$, whereas additive impact corresponds to a LOB shape being specified with respect to absolute price pertubations $S_t-\baseS_t$  as in \cite{PredoiuShaikhetShreve11}.  
 Note that the LOB shape is static (and  \cref{sect:intermediate buy orders} considers a two-sided LOB wth zero bid-ask spread).
 Such can be viewed as a low-frequency model for price impact according to a LOB shape which is representative on  longer horizons, but not for  high frequency trading over short periods.

Let $s = \rho \baseS_t$ be some price close to the unaffected price $\baseS_t$ and let $q(\rho) \diff \rho$ denote the density  of (bid or ask) offers at price level $s$, i.e.~at the relative price perturbation $\rho$. 
This leads to a measure with cumulative distribution function $Q(\rho) := \int_1^\rho q(x) \diff x$, $\rho \in (0,\infty)$. 
The total volume of orders at prices corresponding to perturbations $\rho$ from some
range $  R \subset (0,\infty)$ 
then is $\int_R q(x) \diff x$. 
Selling $-\Delta \assetsProcess_t$ shares at time $t$ shifts the price from $\rho_{t-} \baseS_t$ to $\rho_t \baseS_t$, while the volume change is $Q(\rho_{t-}) - Q(\rho_t) = -\Delta \assetsProcess_t$. 
The proceeds from this sale are $\baseS_t \int_{\rho_t}^{\rho_{t-}} \rho \diff Q(\rho)$. 
Changing variables, with $Y_t := Q(\rho_t)$ and $f := Q^{-1}$, the proceeds can be expressed as in \cref{eq:block sale proceeds}. 
In this sense,  $Y$ from \eqref{eq:deterministic Y_t dynamics} can be understood as the \emph{volume effect process} as in \cite[Section~2]{PredoiuShaikhetShreve11}. See \cref{fig:order-book} for illustration.

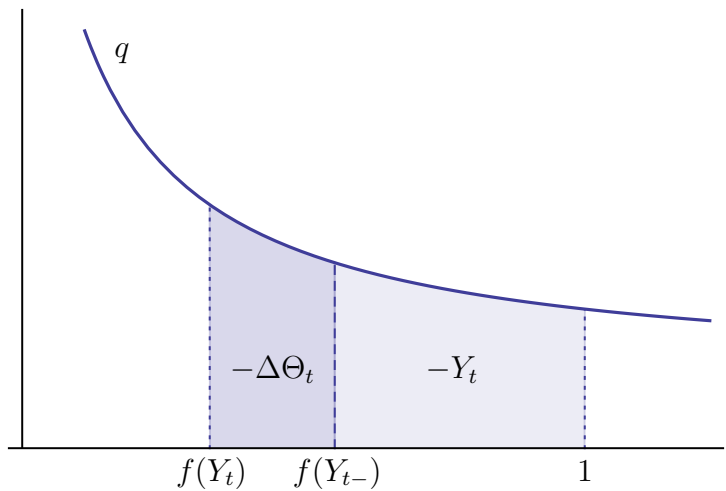
\begin{figure}[ht]
	\centering
	\begin{overpic}[width=0.6\textwidth]{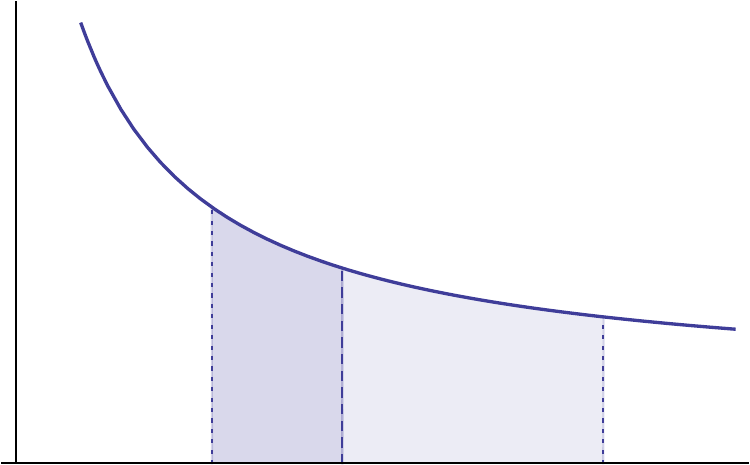}
		\put(15,55){$q$}
		
		\put(37,10){\clap{$-\Delta\assetsProcess_t$}}
		\put(62,10){\clap{$-Y_t$}}
		
		\put(28.5,0){\raisebox{-1em}{\clap{$f(Y_t)$}}}
		\put(46,0){\raisebox{-1em}{\clap{$f(Y_{t-})$}}}
		\put(80.5,0){\raisebox{-1em}{\clap{$1$}}}
	\end{overpic}%
	\vspace{1em}
	\caption{Order book density $q$ and behavior of the multiplicative price impact $f(Y)$ when selling a block of size $-\Delta \assetsProcess_t > 0$. Note that $-Y_t = -Y_{t-} - \Delta \assetsProcess_t$.}
	\label{fig:order-book}
\end{figure}

\begin{example}\label{ex:order book densities}
	Let the (one- or two-sided) shadow limit order book density be $q(x) \!:=\! {c}/{x^r}$ on $x\in (0,\infty)$ for constants $c,r>0$. 
	Parameters $c$ and $r$ determine the market depth (LOB volume): 
	If $r < 1$, a trader can sell only finitely many but buy infinitely many assets at any time. 
	In contrast, for $r>1$ one could sell infinitely many but buy only finitely many assets at any time instant and (by \eqref{eq:deterministic Y_t dynamics}) also in any finite time period. Note that \cite[p.185]{PredoiuShaikhetShreve11} assume infinite market depth in the target trade direction.
	The case $r=1$ describes infinite market depth in both directions.
	The antiderivative $Q$ and its inverse $f$ are determined for  $x > 0$ and $(r-1)y \ne c$ as 
	\begin{align*}
		Q(x) &= \begin{cases} c \log x, &\text{for }r=1, \\ \frac{c}{1-r}(x^{1-r}-1), &\text{otherwise,} \end{cases}%
	&	f(y) = \begin{cases} e^{y/c}, &\text{for }r=1, \\ \paren[\big]{1 + \frac{1-r}{c} y}^{1/(1-r)}, &\text{otherwise.} \end{cases}
	\end{align*}
	For the parameter function $\lambda$ this yields \(\lambda(y) = f'(y)/f(y)= ({c + (1-r)y})^{-1}\).
	Note that for $r\ne 1$ the functions $f$ and $\lambda$ are effectively constrained to the domain $(\frac{c}{r-1}, \infty)$ for $r < 1$ and $(-\infty, \frac{c}{r-1})$ for $r > 1$.
We have assumed in the paper that  $f>0$ is defined on the whole real line for simplicity.
Yet, let us use this example to explain next how also interesting cases like $r\in (0,\infty)\setminus\{1\}$ can be dealt with
 by refining the definition of the set of admissible strategies according to $f$. Indeed, properties of $f$ are only needed within the range of possible values of processes $Y^\Theta$. Hence, the more general case where 
$I_f:=\{y: 0<f(y)<\infty\}$ is an open interval in $\RR$ can be treated by imposing as an additional requirement
for  admissibility
 of a strategy $\Theta$  (in \eqref{eq:admmonostrat}, \eqref{eq:admfvstrat}) that  $Y^\Theta$ has to evolve in $I_f$.
\cref{ex:simple-lob} will investigate this case further.
\end{example}

\section{Optimal~liquidation with monotone strategies}
\label{sec:Problem}

This section solves the optimal liquidation problem that is central for this paper.
The large investor is facing the task to sell $\assetsProcess_{0-}$ risky assets but has the possibility to split it into smaller orders to improve according to some performance criterion. 
Before \cref{sect:intermediate buy orders}, we will restrict ourselves to monotone control strategies that do not allow for intermediate buying. 
The analysis for this more restricted variant of control policies will be shown later in  \cref{sect:intermediate buy orders} to carry over to an alternative problem with a wider set of controls, being of finite variation, admitting also intermediate buy orders.

For an initial position of $\assetsProcess_{0-}$ shares,  the set of admissible trading strategies
 is
\begin{align} \label{eq:admmonostrat}
\begin{split}
	\admissibleSellStrategies{\assetsProcess_{0-}} := \bigl\{ \assetsProcess \bigm| {}&\text{$\assetsProcess$ is decreasing, càdlàg, predictable,} 
\\			&\text{with }\assetsProcess_{0-} \ge \assetsProcess_t \ge 0 \bigr\}.
\end{split}
\end{align}
Here, the quantity $\assetsProcess_t$ represents the number of shares held at time $t$.
Any admissible strategy $\assetsProcess \in \admissibleSellStrategies{\assetsProcess_{0-}}$ decomposes into a continuous and a discontinuous part
\begin{equation}
	\assetsProcess_t = \assetsProcess^c_t + \sum_{0 \le s \le t} \Delta \assetsProcess_s,
\end{equation}
where $\assetsProcess^c_t$ is continuous (and increasing) and $\Delta \assetsProcess_s := \assetsProcess_s - \assetsProcess_{s-} \le 0$. 
Aiming for an explicit analytic solution, we consider trading on the infinite time horizon $[0, \infty)$ with discounting. 
The $\gamma$-discounted proceeds from strategy $\assetsProcess$ up to time $T<\infty$ are
\begin{equation} \label{def:liquidation proceeds}
	L_T(y;\assetsProcess):= -\int_0^T e^{ - \gamma t}f(Y_{t})\baseS_{t} \diff \assetsProcess^c_t - \sum_{\substack{0\leq t \leq T \\ \Delta \assetsProcess_t \neq 0}} e^{ - \gamma t} \baseS_t \int_0^{\Delta \assetsProcess_t} f(Y_{t-} + x) \diff x,
\end{equation}
where $y = Y_{0-}$ is the initial state of process $Y$. 
Clearly, $Y_{0-}$ and $\assetsProcess$ determine $Y$ by \eqref{eq:deterministic Y_t dynamics}.
\begin{remark} \label{rmk:convergence of block proceeds}
The (possibly) infinite sum in \eqref{def:liquidation proceeds} has finite expectation. 
Indeed, for any $\assetsProcess \in \admissibleSellStrategies{\assetsProcess_{0-}}$ one has $\sup_{t\leq T}|Y_t|<\infty$.
Hence, the mean value theorem and properties of $f$ imply for $t\in [0,T]$ that 
\[
	0 \le -\int_0^{\Delta \assetsProcess_t} f(Y_{t-} + x)\diff x 
	\le -\Delta \assetsProcess_t \sup_{x\in (\Delta \assetsProcess_t, 0)} f(Y_{t-} + x) 
	\le -\Delta \assetsProcess_t \cdot \sup_{t\le T} f(Y_{t})
	\,.
\]
Thus, by finite variation of $\assetsProcess$ the infinite sum in \eqref{def:liquidation proceeds} a.s.\ converges absolutely.
For $\assetsProcess \in \admissibleSellStrategies{\assetsProcess_{0-}}$ the sum is bounded in expectation, because $Y$ and hence $\sup_{t\le T}f(Y_t)$ are bounded, and we have $\EE[\sup_{t\in [0,T]} \baseS_t] < \infty$ and $0 \le \sum_{t\in [0,T]} (-\Delta \assetsProcess_s) \le \assetsProcess_{0-}$.
\end{remark}

Note that the monotone limit $L_\infty(y;\assetsProcess) := \lim_{T \nearrow \infty} L_T(y;\assetsProcess)$ always exists. 
We consider the control problem to find the optimal strategy that maximizes the expected (discounted) liquidation proceeds over an open (infinite) time horizon
\begin{align}\label{task:maximize expected liquidation proceeds} 
&	\max_{\assetsProcess \in \admissibleSellStrategies{\assetsProcess_{0-}}} J(y;\assetsProcess) \qquad \text{for} \quad  J(y;\assetsProcess) := \EE[ L_\infty(y;\assetsProcess) ],
\\
& \text{with value function}\quad 
\label{eq:val fn mon case}
	v(y,\theta) := \sup_{\assetsProcess \in \admissibleSellStrategies{\theta}} J(y;\assetsProcess).
\end{align}
 For this problem maximizing over deterministic strategies turns out to be sufficient (see \cref{rmk:optimal strategy is deterministic} below). Since expectations $\EE[\exp(-\gamma t)\baseS_t]=\baseS_0\exp(-t(\gamma-\mu))$, $t\ge 0$, depend on  $\mu,\gamma$ only through $\delta := \gamma - \mu$, for our optimization problem just the difference $\delta$ matters which needs to be positive to have $v(y,\theta)<\infty$ for $\theta>0$.  
Thus, regarding $\gamma$ and $\mu$, only the difference  $\delta$ will be needed, and it might be
interpreted as impatience parameter chosen by the large investor (when choosing $\gamma$),  specifying her preferences to liquidate earlier rather than later, as a drift rate of the risky asset returns $d\baseS/\baseS$, or as a combination thereof.
The following conditions on $\delta, f, h$ are assumed for \cref{sec:Problem,sec:Free boundary problem,sect:intermediate buy orders}.
\begin{assumption}\label{cond:model parameters}
	The map  $t\mapsto \EE[e^{-\gamma t}\baseS_t]$, $t\ge 0$, is decreasing, i.e.\ $\delta:= \gamma - \mu > 0$.\\
	The price impact function $f: \RR \to (0,\infty)$ satisfies $f(0)=1$, $f \in C^2$ and is strictly increasing such that $\lambda(y):= f'(y)/f(y) > 0$ everywhere.\\
	The resilience function $h: \RR \to \RR$ from \eqref{eq:deterministic Y_t dynamics} is $C^2$ with $h(0) = 0$ and $h' > 0$.\\	
	Resilience and market impact satisfy $(h \lambda)' > 0$ and $(h\lambda + h')' > 0$.\\
	There exist solutions $y_0$ to $h(y_0)\lambda(y_0) + \delta = 0$ and $y_\infty$ to $h(y_\infty)\lambda(y_\infty) + h'(y_\infty) + \delta = 0$. (Uniqueness of $y_0$ and $y_\infty$ holds by the other conditions.)
\end{assumption}

\begin{remark}[Interplay of impact and resilience functions]
The two assumptions $(h\lambda)' > 0$ and $(h\lambda + h')' > 0$ are technical requirements for our verification of optimality.
Already from the shadow limit order book (LOB) perspective (cf.\ \cref{sect:LOB}), some sort of condition connecting both resilience speed $h$ and price impact $f$, thus LOB shape, appears natural.
Examples satisfying \cref{cond:model parameters} with $f(y):= e^{\lambda y}$ for constant $\lambda > 0$ are e.g.~linear resilience speed, $h(y) = \beta y$ with $\beta > 0$ and any discounting $\delta > 0$; or for instance bounded resilience,
\(
	h(y) = \alpha \arctan(\beta y),
\)
for $\alpha,\beta > 0$, with $\beta < \lambda$ and not too large discounting $0 < \delta < \frac{1}{2}\alpha \lambda \pi$ (a larger $\delta$ would give that the trivial strategy to sell everything initially at time $0$ is optimal).
\end{remark}

The main \cref{thm: optimal strategy,,thm:optimal liquidation with intermediate buy orders}  solve the optimal liquidation problem for one- respectively two- sided limit order books.
The proof of \cref{thm: optimal strategy} is given in \cref{sec:Free boundary problem}.

\begin{theorem}\label{thm: optimal strategy}
	Let the model parameters $h$, $\lambda$, $\delta$ satisfy \cref{cond:model parameters} and $\assetsProcess_{0-} \ge 0$ be given.
	Define ${y_\infty < y_0 < 0}$ as the unique solutions of ${h(y_\infty)\lambda(y_\infty) + h'(y_\infty) + \delta = 0}$ and ${h(y_0) \lambda(y_0) + \delta = 0}$, respectively, and let
	\begin{equation} \label{defThmOptimalStrategyTTL}
		\tau(y):= -\frac{1}{\delta} \log \paren[\bigg]{\frac{f(y)}{f(y_0)} \frac{h(y)\lambda(y) + h'(y) + \delta}{h'(y)}},
	\end{equation}
	for ${y \in (y_\infty, y_0]}$ with inverse function ${\tau \mapsto \ybar(\tau) : [0,\infty) \to (y_\infty, y_0]}$.
	Moreover, let $\theta(y)$, ${y \in (y_\infty, y_0]}$, be the strictly decreasing solution to the ordinary differential equation
	\begin{equation} \label{defThmOptimalStrategyDiffThetaY}
	\begin{aligned}
		\theta'(y) &= 1 + \frac{h(y) \lambda(y)}{\delta} - \frac{h(y)h''(y)}{\delta h'(y)} + \frac{h(y) \paren[\big]{h\lambda + h' + \delta}'(y)}{\delta \paren[\big]{h\lambda + h' + \delta}(y)},\quad y \in (y_\infty, y_0]\,,
	\end{aligned}
	\end{equation}
	with initial condition $\theta(y_0) = 0$, and let $\theta \mapsto y(\theta)$, $\theta \ge 0$,  denote its inverse.
	For given $\assetsProcess_{0-}$ and $Y_{0-}$, define the liquidation strategy $\assetsProcess = \assetsProcess^\text{opt}$ as follows.
	\begin{enumerate}
		\item\label{strategy-step:sell all}
			If $Y_{0-} \ge y_0 + \assetsProcess_{0-}$, sell all assets at once: $\assetsProcess_0 = 0$.
		
		\item\label{strategy-step:sell block}
			If $y(\assetsProcess_{0-}) < Y_{0-} < y_0 + \assetsProcess_{0-}$, then sell a block of size $-\Delta \assetsProcess_0 \equiv \assetsProcess_{0-} - \assetsProcess_0$ such that $\assetsProcess_0 > 0$ and $Y_0 \equiv Y_{0-} + \Delta \assetsProcess_0 = y(\assetsProcess_0)$.
		
		\item\label{strategy-step:wait}
			If $Y_{0-} < y(\assetsProcess_{0-})$, wait until time $s = \inf \{ t>0 \mid y_w(t) = y(\assetsProcess_{0-})\} < \infty$, where $y_w$ is the solution to the ODE $y_w'(t) = -h(y_w(t))$ with initial condition $y_w(0) = Y_{0-}$.
			That is, set $\assetsProcess_t = \assetsProcess_{0-}$ for $t < s$.
			This leads to $Y_t = y_w(t)$ for $t < s$.
		
		\item\label{strategy-step:follow boundary}
			As soon as step~\ref{strategy-step:sell block} or~\ref{strategy-step:wait} lead to the state $Y_s = y(\assetsProcess_s)$ for some time $s \ge 0$, sell continuously: $\assetsProcess_t = \theta(\ybar(T-t))$, $s \le t \le T$, until time $T = s + \tau(y(\assetsProcess_s))$.
		
		\item\label{strategy-step:finished}
			Stop when all assets are sold
			at some time $T<\infty$: $\assetsProcess_t = 0$, $t \in [T,\infty)$.
	\end{enumerate}
	Then the strategy $\assetsProcess^\text{opt}$ is the unique maximizer to the problem~\eqref{task:maximize expected liquidation proceeds} of optimal liquidation
	\(
		\max_{\assetsProcess \in \admissibleSellStrategies{\assetsProcess_{0-}}} \EE[ L_\infty(y;\assetsProcess) ]
	\)
	for $\assetsProcess_{0-}$ assets with initial market impact being $Y_{0-} = y$.
\end{theorem}

The optimal liquidation strategy is deterministic.
Note that it does not depend on the particular form of the martingale $M$ (what has been noted as a robust property  in related literature).
Since $T<\infty$ is finite, the open horizon control from \cref{thm: optimal strategy} is clearly optimal for the problem on any finite horizon $T'\ge T$; cf.\ \cref{rmk: finite time horizon; no drift; convexity} for $T'<T$.

\begin{remark}
	\cite{PredoiuShaikhetShreve11} consider a similar optimal execution problem, with an additive price impact $\psi$ such that $S_t = \baseS_t + \psi(Y_t)$ with volume effect process $Y_t$ as in \eqref{eq:deterministic Y_t dynamics}. 
	They study the case of martingale $\baseS_t$ on a finite time horizon $[0,T]$.
	The execution costs, which they seek to minimize in expectation, are equal to the negative liquidation proceeds $-L_T$ in our model (for $\gamma,\mu=0$) with fixed $Y_{0-} := 0$.
	See also \cref{rmk: finite time horizon; no drift; convexity} below.
\end{remark}

The next result provides sufficient conditions for optimality  to the problem \eqref{task:maximize expected liquidation proceeds} for each possible initial state $Y_{0-} = y\in \RR$ of the impact process, by the martingale optimality principle. 
In contrast, in the related additive model in \cite{PredoiuShaikhetShreve11} the optimal buying strategy for finite time horizon without drift ($\delta=0$), and impact process starting at zero was characterized using an elegant convexity argument; cf.\ \cref{rmk: finite time horizon; no drift; convexity}.

\begin{proposition}\label{prop: supermatringale suffices}
	Let $V:\RR \times [0,\infty) \to [0,\infty)$ be a continuous function such that $G_t(y;\assetsProcess) := L_t(y;\assetsProcess) + e^{-\gamma t} \baseS_t \cdot V(Y_t,\assetsProcess_t)$, with $Y = Y^\assetsProcess$ and $y=Y_{0-}$, is a supermartingale for each $\assetsProcess \in  \admissibleSellStrategies{\assetsProcess_{0-}}$ and additionally $G_0(y;\assetsProcess) \le G_{0-}(y;\assetsProcess):= \baseS_0 \cdot V(Y_{0-}, \assetsProcess_{0-})$. 
	Then 
	\[
		\baseS_0 \cdot V(y, \theta) \geq v(y,\theta)
	\]
	with $\theta = \assetsProcess_{0-}$.
	Moreover, if there exists $\assetsProcess^{*} \in \admissibleSellStrategies{\assetsProcess_{0-}}$ such that $G(y;\assetsProcess^{*})$ is a martingale and it holds $G_0(y;\assetsProcess^{*}) = G_{0-}(y;\assetsProcess^{*})$, then 
	\(
		\baseS_0 \cdot V(y, \theta) = v(y,\theta)
	\)
	and $v(y,\theta) = J(y;\assetsProcess^{*}).$
\end{proposition}
\begin{remark}\label{rem:supermartprop}
	The additional condition on $G_0$ and $G_{0-}$ can be regarded as extending the (super-)martingale property from time intervals $[0,T]$ to time \enquote{$0-$}.
\end{remark}
\begin{proof}
	Note that $\EE\brackets[\big]{G_{0-}(y;\assetsProcess)} = G_{0-}(y;\assetsProcess) = \baseS_0 \cdot V(y, \theta)$ and 
	\[
		\EE\brackets[\big]{ G_t(y;\assetsProcess) } = \EE\brackets[\big]{ L_t(y;\assetsProcess) } + \EE\brackets[\big]{ e^{-\gamma t}\baseS_t\cdot V(Y_t, \assetsProcess_t) }
	\]
	for each $t\geq 0$.
	Also, $V(Y_t, \assetsProcess_t)$ is bounded uniformly on $t\geq 0$ and $\assetsProcess \in \admissibleSellStrategies{\assetsProcess_{0-}}$ by a finite constant $C > 0$, because $V$ is assumed to be continuous (and hence bounded on compacts) and the state process $(Y,\assetsProcess)$ takes values in the rectangle
	\(
		\brackets[\big]{ -\abs{y} - \theta, \abs{y} + \theta } \times \brackets[\big]{ 0, \theta }.
	\)
	Hence, 
	\(
		\EE\brackets[\big]{ e^{-\gamma t}\baseS_t\cdot V(Y_t, \assetsProcess_t) } \leq C e^{-\gamma t} \EE\brackets[\big]{ \baseS_t } = C e^{-\delta t} \baseS_0
	\) tends to  $0$ for $t \to \infty$, since $\delta > 0$.
	Since $\EE\brackets{ L_t(y;\assetsProcess) } \to J(y;\assetsProcess)$ as $t \to \infty$ by means of monotone convergence theorem, we conclude that
	\(
		\baseS_0\cdot V(y,\theta)\geq G_0(y;\assetsProcess) \ge \EE\brackets{ G_t(y;\assetsProcess) } \to J(y; \assetsProcess).
	\)
	This implies the first part of the claim. 
	The second part follows analogously.
\end{proof}

\noindent
In order to make use of \cref{prop: supermatringale suffices}, one applies Itô's formula to $G$, assuming that $V$ is smooth enough and using the fact that $[\baseS_\cdot, e^{-\gamma \cdot}V(Y_\cdot, \assetsProcess_\cdot)] = 0$ because $\baseS$ is quasi-left-continuous and $e^{-\gamma \cdot}V(Y_\cdot, \assetsProcess_\cdot)$ is predictable and of bounded variation, to get
\begin{align}\label{eq: Ito for G}
\begin{split}
	\diff G_t = {}&e^{-\delta t} V(Y_{t-}, \assetsProcess_{t-}) \diff M_t 
	\\		& + e^{-\delta t} M_{t-} \biggl( \paren[\big]{-\delta V - h V_y}(Y_{t-}, \assetsProcess_{t-}) \diff t 
	\\		&\qquad\qquad\quad + \paren[\big]{V_y + V_\theta - f}(Y_{t-}, \assetsProcess_{t-}) \diff \assetsProcess^c_t 
	\\		&\qquad\qquad\quad + \int_0^{\Delta \assetsProcess_t}  \paren[\big]{V_y + V_\theta - f}(Y_{t-} + x, \assetsProcess_{t-} + x) \diff x \biggr)
\end{split}
\end{align}
with the abbreviating conventions ${\paren[\big]{-\delta V - h V_y}(a,b) := -\delta V(a,b) - h(a) V_y(a,b)}$ and $\paren[\big]{V_y + V_\theta - f}(a,b) := V_y(a,b) + V_\theta(a,b) - f(a)$.
The martingale optimality principle now suggests equations for regions where the optimal strategy should sell or wait, in that the $\diff \assetsProcess$-integrands should be zero when there is selling and the $\diff t$-integrand must vanish when only time passes (waiting).
We will construct a classical solution to the \emph{variational inequality} 
$\max \{-\delta V - h V_y\,,\, f-\ V_y - V_\theta  \}=0$, that is
a  function $V$ in $C^{1,1}(\RR \times [0,\infty),\RR)$ and a strictly decreasing \emph{free boundary} function $y({\cdot}) \in C^2( [0, \infty),\RR)$, such that
\begin{align}
	-\delta V - h(y) V_y &= 0 &&\hbox{in } \overbar{\scW} \label{eq:var ineq I}
\\ -\delta V - h(y) V_y &< 0 &&\hbox{in  }\ \scS  \label{eq:var ineq II}
\\ V_y + V_\theta &= f(y) &&\hbox{in  }\ \overbar{\scS} \label{eq:var ineq III}
\\ V_y + V_\theta &> f(y) &&\hbox{in }  \scW \label{eq:var ineq IV}
\\ V(y,0) &= 0 &&\forall y\in \RR \label{eq:var ineq boundary cond}
\end{align}
for wait region $\scW$ and sell region $\scS$ (cf. \cref{fig:wait-and-sell-region}) defined as
\begin{align}
\begin{split}
	\scW &:= \{(y,\theta)\in \RR \times [0,\infty) \mid y < y(\theta)\},\\
	\scS &:= \{(y, \theta)\in\RR\times[0,\infty) \mid y> y(\theta)\}.
\end{split}
\end{align}
The optimal liquidation studied here belongs to the class of finite-fuel control problems, which often lead to \emph{free boundary problems} similar to the one derived above. 
See \cite{KaratzasShreve86} for an explicit solution of the finite-fuel monotone follower problem, and \cite{JackJohnsonZervFos08} for further examples and an extensive list of references. 
In the next section, we construct the (candidate) boundary $y(\theta)$ and then the value function $V$, and prove that 
they  solve the desired equations and that the derived control strategy is optimal.

\begin{remark}[On notation]
	We have three a priori independent dimensions at hand: The time $t$, the investor's holdings $\theta$ and her market impact $y$. 
	To assist  intuition, we will overload notation by writing $y(\theta)$ or $y(t)$ for the $y$-coordinate as a function of holdings or of time along the boundary between $\scS$ and $\scW$, instead of introducing various function symbols for the relation between these coordinates. 	Accordingly, the inverse function of $y(\theta)$ is $\theta(y)$. 
		The advantage  is that readers can identify the meaning of individual terms at a glance, 
but have to keep in mind that these are different functions, 
e.g.~when differentiating. 
\end{remark}

\begin{remark}[On deterministic optimal controls]\label{rmk:optimal strategy is deterministic}
We will obtain that optimal strategies are deterministic and the value function is continuous (even differentiable).
This is shown in the subsequent sections by proving \cref{thm: optimal strategy}.
Here, we show directly why non-deterministic strategies are suboptimal for \eqref{eq:val fn mon case} and optimizing over deterministic admissible controls is sufficient.
Yet, finding explicit solutions here
 still requires to construct candidate solutions and prove optimality, as in the sequel.
 
If one considers optimization just over strategies that are to be executed until a time $T < \infty$, then the value function will be the same as if we were optimizing over the subset of deterministic strategies.
Indeed, by optional projection  (see~\cite[VI.57]{DellacherieMeyer82bookB}) we have
\begin{align*}
	\EE[L_T(y;\assetsProcess)] = -\EE\brackets[\Big]{ M_T \int_0^T \! e^{ - \delta t}f(Y_{t-}) \diff \assetsProcess^c_t } - \EE\brackets[\Big]{ M_T \sum_{\mathclap{\substack{0\leq t \leq T \\ \Delta \assetsProcess_t \neq 0}}} e^{ - \delta t}  \int_0^{\Delta \assetsProcess_t} \! f(Y_{t-} + x) \diff x }.
\end{align*}
For any  $T\in [0,\infty)$, letting $\diff \widetilde{\PP} = M_T/ M_0\diff \PP$ on $\scF_T$  yields that $\EE[L_T(y;\assetsProcess)]$ equals $\EE^{\widetilde{\PP}}[ \ell_T(\assetsProcess) ]$ for
\(	\ell_T(\assetsProcess) := -M_0 \int_0^T e^{ - \delta t}f(Y_{t-}) \diff \assetsProcess^c_t - M_0\sum_{0\leq t \leq T\, , \,\Delta \assetsProcess_t \neq 0} e^{ - \delta t}  \int_0^{\Delta \assetsProcess_t} f(Y_{t-} + x).
\)
Note that $\ell$ is a deterministic functional of $\assetsProcess$, and that the measure $\widetilde{\PP}$ does not depend on $\assetsProcess$. 
Thus, optimization for any finite horizon $T$ can be done $\omega$-wise, i.e.\ for the finite-horizon problem optimizing over the subset of  deterministic strategies gives the same value function.
Note that this is similar to \cite[Prop.~7.2]{Lokka12}. 
Using monotonicity of  $L_T$ in $T$, we have $\EE[L_{\infty}(y;\assetsProcess)]$ $= \sup_{T\in [0,\infty)}\EE[L_{T}(y;\assetsProcess)]$, 
hence the change of measure argument above yields that  
$v(y,\theta)
 =\sup_{T \in [0,\infty)}\sup_{\assetsProcess \in \admissibleSellStrategies{\theta}} \EE[ L_T(y;\assetsProcess) ] $
is equal to
\begin{align}
	 \sup_{T \in [0,\infty)}\sup_{\substack{\assetsProcess \in \admissibleSellStrategies{\theta} \\ \text{deterministic}}}\ell_T(\assetsProcess) %
	= \sup_{\substack{\assetsProcess \in \admissibleSellStrategies{\theta} \\ \text{deterministic}}}\ell_\infty(\assetsProcess) \label{eq:determ probl inf horizon}.
\end{align}
Moreover, one can check that any deterministic maximizer  $\assetsProcess^*\in  \admissibleSellStrategies{\theta}  $ to \eqref{eq:determ probl inf horizon} is also  optimal for the original problem \eqref{task:maximize expected liquidation proceeds}, where  $v(y,\theta)<\infty$ thanks to $\delta<0$.

\end{remark}

\begin{remark}%
\label{rmk: finite time horizon; no drift; convexity}
For a given finite horizon $T<\infty$, the execution problem with general order book shape has been solved by \cite{PredoiuShaikhetShreve11} for additive price impact and no drift ($\delta=0$).  
The problem with multiplicative impact could be transformed to the additive situation using  intricate state-dependent order book shapes, cf.~\cite{Lokka12}.
Let us show how a convexity argument as in \cite{PredoiuShaikhetShreve11} can be applied also directly to solve the finite horizon case in the multiplicative setup when the drift $\delta$ is zero, but not for $\delta \ne 0$. 

By \cref{rmk:optimal strategy is deterministic} it suffices to consider deterministic strategies $\assetsProcess \in \admissibleSellStrategies{\assetsProcess_{0-}}$.
Let $F(y) = \int_0^y f(x) \diff x$. 
For deterministic $\assetsProcess$ and $g(x):= f(h^{-1}(x))x + \delta F(h^{-1}(x))$ we have
\begin{equation} \label{eq: deterministic proceeds; finite time horizon with drift}
	\EE[L_T(\assetsProcess)] = F(Y_{0-}) - e^{\delta T} F(Y_T) - \int_0^T e^{-\delta t} g(h(Y_t)) \diff t\,.
\end{equation}
Since $g'(x) = \bp{ f (h\lambda + h' + \delta) / h' }(h^{-1}(x))$, the function $g$ obtains global maximum at $h(y_\infty)$ and is decreasing on the left and increasing on the right.
So its convex hull $\hat g(x) = \{ \ell(x) \mid \ell \le g \text{ is affine} \}$ exists.
With $C_{\delta,T} := \int_0^T e^{-\delta t} \diff t$ Jensen's inequality yields
\begin{equation} \label{ineq: deterministic proceeds; finite time horizon with drift; Jensen}
	\EE[L_T(\assetsProcess)] \le F(Y_{0-}) - e^{\delta T} F(Y_T) - C_{\delta,T} \cdot \hat g\paren[\bigg]{ \int_0^T h(Y_t) \frac{e^{-\delta t}}{C_{\delta,T}} \diff t }
\end{equation}
and we have equality in \eqref{ineq: deterministic proceeds; finite time horizon with drift; Jensen} if and only if $h(Y_t)$ stays constant in the interval where $\hat g$ and $g$ coincide for almost all $t \in (0,T)$.
Such impact-fixing strategies are analogous to the \emph{type~A strategies} of \cite{PredoiuShaikhetShreve11}.

The integral in \eqref{ineq: deterministic proceeds; finite time horizon with drift; Jensen} can be solved in general only for $\delta = 0$.
In that case $C_{\delta,T} = T$ and $\int_0^T h(Y_t) \diff t = Y_{0-} - \assetsProcess_{0-} - Y_T$ for any strategy that liquidates until time $T$, so
\begin{equation}
	\EE[L_T(\assetsProcess)] \le F(Y_{0-}) - F(Y_T) - T \hat g\paren[\bigg]{ \frac{Y_{0-} - \assetsProcess_{0-} - Y_T}{T} }=: \hat G(Y_T)\,.
\end{equation}
Let $e(y):= (Y_{0-} - \assetsProcess_{0-} - y) / T$.
Direct calculations show that the concave function $\hat G$ is decreasing at $e^{-1}(h(y_\infty))$ and (if $Y_{0-} \le \assetsProcess_{0-}$) non-decreasing at $e^{-1}(h(y_0))$.
So $\hat G$ obtains a global maximum at some $y^* \in [e^{-1}(h(y_0)), e^{-1}(h(y_\infty))]$.
Straight forward calculations show that $\hat g = g$ on $[h(y_\infty), h(y_0)]$.

Consider the \emph{type~A strategy} $\assetsProcess^*$ which performs an initial block trade $\Delta \assetsProcess^*_0 = Y_0 - Y_{0-}$ to reach impact level $Y_0 := h^{-1}\bp{ e(y^*) }$, trades continuously at rate $\!\diff \assetsProcess^*_t/\!\diff t = h(Y_0) = e(y^*)$ until time $T$ and finishes with a block trade of size $\Delta\assetsProcess^*_T = y^* - Y_0$ reaching impact level $Y_T = y^*$.
By construction, we have $\EE[L_T(\assetsProcess^*)] = \hat G(y^*)$, so $\assetsProcess^*$ is optimal (if $Y_{0-} \le \assetsProcess_{0-}$).
\end{remark}

\section{Solving the free boundary problem}\label{sec:Free boundary problem}

In the next two subsections, we construct an explicit solution to our free boundary problem of finding $\widebar \scW \cap \widebar \scS = \{ (y(\theta),\theta) \mid \theta \ge 0 \} = \{ (y,\theta(y)) \mid \dots \}$.
We will find that under \cref{cond:model parameters} the optimal strategy is described by the free boundary with
\begin{equation}\label{eq:theta' common}
	\theta'(y) = 1 + \frac{h(y) \lambda(y)}{\delta} - \frac{h(y)h''(y)}{\delta h'(y)} + \frac{h(y) \paren[\big]{h\lambda + h' + \delta}'(y)}{\delta \paren[\big]{h\lambda + h' + \delta}(y)}
\end{equation}
for $y$ in some appropriate interval $(y_\infty, y_0]$ and $\theta(y_0) = 0$, see \cref{fig:wait-and-sell-region} for a graphical visualization.
In \cref{sec:prop of the candidate boundary} we verify that \eqref{eq:theta' common} defines a monotone boundary with a vertical asymptote, and in \cref{subsection: Construction of V and the optimal strategy} we construct $V$ solving the free boundary problem \eqref{eq:var ineq I} -- \eqref{eq:var ineq boundary cond}, completing the verification of the optimal liquidation problem.

\subsection{Smooth-pasting approach}\label{subsection: smooth-pasting}

Following the literature on finite-fuel stochastic control problems, cf.\ e.g.~\cite[Section~6]{KaratzasShreve86}, we apply  the principle of smooth fit to derive a candidate boundary given by \eqref{eq:theta' common} dividing the sell region and the wait region.
To this end, let us at first assume that a solution $(V, y(\cdot))$ is already constructed and is sufficiently smooth along the free boundary.
This will permit to derive by algebraic arguments the (candidate) free boundary and the function $V$ on it. 
\cref{subsection: Construction of V and the optimal strategy} will verify that this approach provides indeed the construction of a classical solution to the free boundary problem.

The first guess we make is that the wait region $\overbar{\scW}$ is contained in $\{(y, \theta):\ y < c\}$ for some $c < 0$.
In this case, the %
solution to \eqref{eq:var ineq I} in the wait region would be of the form 
\begin{equation}\label{eq:V in wait region}
	V(y,\theta) = C(\theta) \exp\paren[\bigg]{ \int_c^y \frac{-\delta}{h(x)}\diff x }, \qquad (y,\theta)\in \overbar{\scW},
\end{equation}
where ${C:[0,\infty) \to [0,\infty)}$.
To shorten further terms, let
\(
	\phi(y):= \exp\paren[]{\int_c^y \frac{-\delta}{h(x)}\diff x}, \; y \le c.
\)
Suppose that $C$ is continuously differentiable.
Calculating the directional derivative ${V_y + V_\theta}$ and the expression ${V_{\theta y} + V_{yy}}$ 
in the wait region,
 we obtain for ${(y, \theta)\in \scW}$ that
\begin{align}
	V_y(y,\theta) + V_\theta(y,\theta) &= -\delta C(\theta)\phi(y)/h(y)  + C'(\theta)\phi(y),
\\	V_{\theta y}(y,\theta) + V_{yy}(y,\theta) &=  -\delta C'(\theta) \phi(y)/h(y) + \delta C(\theta) \phi (y)h^{-2}(y)(\delta + h'(y)).
\end{align}
On the other hand, the same expressions computed in the sell-region yield (for $(y,\theta)\in \scS$)
\begin{align*}
	V_y(y,\theta) + V_\theta(y,\theta) = f(y)\quad \text{and} \quad	V_{\theta y}(y,\theta) + V_{yy}(y,\theta) = f'(y).
\end{align*}
Now, suppose that $V$ is a $C^{2,1}$-function.
In particular, we must have for $y = y(\theta)$:
\begin{equation}\label{eq:C^2 smooth pasting boundary}
\left\{
	\begin{aligned}
		f(y) &=  -\delta C(\theta)\phi(y)/h(y)  + C'(\theta)\phi(y),
	\\	f'(y) &=  -\delta C'(\theta) \phi(y)/h(y) + \delta C(\theta) \phi (y)h^{-2}(y)(\delta + h'(y)).
	\end{aligned} \right.
\end{equation}
Solving \eqref{eq:C^2 smooth pasting boundary} as a linear system for $C(\theta)$ and $C'(\theta)$, we get at $y= y(\theta)$:
\begin{equation}\label{eq:C on the boundary}
\left\{
	\begin{aligned}
		C(\theta) &=  f(y)\cdot\frac{1}{\phi(y)}\cdot h(y)\frac{\delta + h(y)\lambda(y)}{\delta h'(y)} &=: M_1(y),
	\\	C'(\theta) &=  f(y)\cdot\frac{1}{\phi(y)}\cdot \frac{\delta + h(y)\lambda(y) + h'(y)}{h'(y)} &=: M_2(y).
	\end{aligned} \right.
\end{equation}
Now, \eqref{eq:C on the boundary} means that we should have
\(
	C(\theta(y)) = M_1(y) \quad\text{and}\quad C'(\theta(y)) = M_2(y),
\)
on the boundary, with $\theta(\cdot)$ being the inverse function of $y(\cdot)$ (in domains of definition to be specified later).
By the chain rule, we get 
\(
	M_1'(y) = C'(\theta(y)) \cdot \theta'(y),
\)
and therefore 
\begin{equation}\label{eq:boundary derivative}
	\theta'(y) = \frac{M_1'(y)}{M_2(y)} = \paren[\bigg]{ \frac{\squeeze[2]{(\delta + 2h\lambda)(h')^2 + (\delta^2 + 2\delta h\lambda + h^2\lambda^2 + h^2 \lambda')h' - h(\delta + h\lambda)h''}}{\delta h'(\delta + h\lambda + h')}}\!(y)
\end{equation}
whenever $\theta(\cdot)$ is defined.
Note that the right-hand sides of \eqref{eq:theta' common} and \eqref{eq:boundary derivative} are equal.

To derive the domain of definition of $\theta(\cdot)$, we use the boundary condition \eqref{eq:var ineq boundary cond} together with \eqref{eq:V in wait region} and \eqref{eq:C on the boundary} to get that $y_0:= y(0) = \theta^{-1}(0)$ solves $\delta + h(y_0)\lambda(y_0) = 0$.
The denominator in \eqref{eq:boundary derivative} suggests that $y_\infty$ solving $\delta + h(y_\infty)\lambda(y_\infty) + h'(y_\infty) = 0$ is a vertical asymptote of the boundary.
Note that \cref{cond:model parameters} implies that $y_\infty < y_0 < 0$ and in particular, we may chose $c\in (y_0, 0)$ at the beginning of this section.
The discussion so far suggests to define a candidate boundary as follows: for $y\in (y_\infty, y_0]$ set
\begin{equation}\label{eq:def of boundary}
	\theta(y) := -\int_y^{y_0} \paren[\bigg]{ \frac{\squeeze[2]{ (\delta + 2h\lambda)(h')^2 + (\delta^2 + 2\delta h\lambda + h^2\lambda^2 + h^2 \lambda')h' - h(\delta + h\lambda)h'' }}{\delta h'(\delta + h\lambda + h')} }\!(x) \diff x.
\end{equation}
We verify in \cref{lemma: boundary} that \eqref{eq:def of boundary} defines a decreasing boundary with $\lim_{y \searrow y_\infty} \theta(y) = +\infty$ and $\theta(y_0) = 0$.
Having a candidate boundary, we can construct $V$ in the wait region $\overbar{\scW}$ in the form \eqref{eq:V in wait region} using \eqref{eq:C on the boundary}, and in the sell region $\scS$ using the directional derivative \eqref{eq:var ineq III}.
In \cref{subsection: Construction of V and the optimal strategy} we prove  that this construction gives a solution to the free boundary problem \eqref{eq:var ineq I} -- \eqref{eq:var ineq boundary cond} and, consequently, to the optimal control problem.

\subsection{Calculus of variation approach}
In this section, we present another approach for finding a candidate optimal boundary by means of calculus of variations. 
Moreover, this gives an explicit description of the time to liquidation along that boundary via~\cref{eq:euler lagrange simplified}, which is not available with the smooth pasting approach above.
To describe the task of finding the optimal boundary as 
an  isoperimetric problem from calculus of variations, we postulate that the optimal strategy is deterministic
(so may assume w.l.o.g.\ $M_t/M_0=1$, cf.\  \cref{rmk:optimal strategy is deterministic})
and that it will liquidate all $\assetsProcess_{0-}$ risky assets within finite time $T:= \inf \{ t \ge 0 \mid \assetsProcess_t = 0 \} < \infty$.
It will be convenient to consider the time to complete liquidation (TTL)
\(
	\tau = T - t
\)
and search for a strategy $\assetsProcess_t = \thetabar(\tau)$ along the boundary $(\ybar(\tau), \thetabar(\tau)) \in \widebar \scW \cap \widebar \scS$, assuming $C^1$-smoothness of that boundary.
By~\eqref{eq:deterministic Y_t dynamics} we have 
\begin{equation}
	\thetabar'(\tau) = \ybar'(\tau) - h(\ybar(\tau)) \label{eq:diff thetabar}
\end{equation}
for the 
function $\ybar(\tau) = y(t) = Y_t$.
So the optimization problem \eqref{task:maximize expected liquidation proceeds} translates to finding $\ybar : [0,\infty) \to \RR$ which maximizes
\begin{align}
	&\widebar J(\ybar) := \int_0^T f(\ybar(\tau)) e^{-\delta (T-\tau)} \paren[\big]{ \ybar'(\tau) - h(\ybar(\tau)) } \diff \tau \label{task:maximize J(y)}
		=: \int_0^T \widebar F(\tau, \ybar(\tau), \ybar'(\tau)) \diff \tau \\
\intertext{with subsidiary condition}
	&\theta = \widebar K(\ybar) := \int_0^T  \paren[\big]{ \ybar'(\tau) - h(\ybar(\tau)) } \diff \tau =: \int_0^T \widebar G(\tau, \ybar(\tau), \ybar'(\tau)) \diff \tau 
\end{align}
for fixed position $\theta:= \assetsProcess_{0-}$.
The Euler equation of this isoperimetric problem is
\begin{equation}
	\widebar F_{\ybar} - \frac{\diff}{\diff \tau} \widebar F_{\ybar'} + \lagrangeMult \paren[\Big]{ \widebar G_{\ybar} - \frac{\diff}{\diff t} \widebar G_{\ybar'} } = 0 \label{eq:euler lagrange general}
\end{equation}
with Lagrange multiplier $\lagrangeMult = \lagrangeMult(T)$.
However, terminal time $T = T(\theta)$, final state $\ybar(0)$ and initial state $\ybar(T)$ are still unknown.
A priori, the final state $\ybar(0)$ is free, which leads to the \emph{natural boundary condition}
\begin{equation}\label{eq:natural boundary condition}
	\Bigl. \widebar F_{\ybar'} + \lagrangeMult \widebar G_{\ybar'} \Bigr|_{\tau=0} = 0.
\end{equation}
With $y_0 := \ybar(0)$ and $y:= \ybar(\tau)$, \cref{eq:natural boundary condition} simplifies to
$\lagrangeMult = -f(y_0) e^{-\delta T}$, and 
\begin{align}\label{eq:euler lagrange simplified}
	 0 &= f(y_0) h'(y) - f(y) e^{\delta \tau} \paren[\big]{ h(y)\lambda(y) + h'(y) + \delta }
\end{align}
follows from  \cref{eq:euler lagrange general}.
Note that the terms involving $\ybar'$ appearing in $\widebar F_{\ybar}$ and $\frac{\diff}{\diff \tau} \widebar F_{\ybar'}$ cancel each other.
Solutions $\ybar_1$ and $\ybar_2$ for time horizons (TTL) $T_1 < T_2$ should coincide for $\tau \in [0,T_1]$, because an optimal (Markov) strategy should depend only on the current position ${\theta = \thetabar(T)}$ and market impact $\ybar(T)$, but not on the past.
In particular, $y_0$ is independent of $T$.
So for $\tau = 0$ we get $h(y_0) \lambda(y_0) + \delta = 0$, justifying the notation $y_0$ as in \cref{cond:model parameters}.
Existence and uniqueness of such $y_0$ is guaranteed by \cref{cond:model parameters}.
It must hold that $y_0 < 0$, because $\lambda > 0$ and $h(y) < 0 \Leftrightarrow y < 0$.
Rearranging \eqref{eq:euler lagrange simplified} gives an explicit description for the time to liquidation along the boundary:
\begin{equation} \label{eq: exp TTL of y}
	e^{-\delta \tau} = \frac{f(y)}{f(y_0)} \frac{h(y)\lambda(y) + h'(y) + \delta}{h'(y)}.
\end{equation}
This defines $\tau \mapsto \ybar(\tau)$ implicitly.
Together with $\thetabar(\tau) = \int_0^\tau \paren[\big]{ \ybar'(\tau) - h(\ybar(\tau)) } \diff \tau$, this function describes the free boundary as a parametric curve.
Differentiating \cref{eq:euler lagrange simplified} with respect to $\tau$, we get
\begin{align}
	\begin{split}
		0 &= f(y_0) h''(\ybar(\tau)) \ybar'(\tau)
		 - f'(\ybar(\tau)) \ybar'(\tau) e^{\delta \tau} \paren[\big]{ h\lambda + h' + \delta }(\ybar(\tau))
	\\		&\quad - \delta f(\ybar(\tau)) e^{\delta \tau} \paren[\big]{ h\lambda + h' + \delta }(\ybar(\tau))
		 - f(\ybar(\tau)) e^{\delta \tau} \paren[\big]{ h'\lambda + h\lambda' + h'' }(\ybar(\tau)) \ybar'(\tau).
\nonumber
	\end{split}
\end{align}
Thus, for $y=\ybar(\tau)$ we obtain
\begin{align}
	\ybar'(\tau) &= \frac{\delta f(y) \paren[\big]{ h\lambda + h' + \delta }(y)}{f(y_0) h''(y) e^{-\delta\tau} - f'(y) \paren[\big]{ h\lambda + h' + \delta }(y) - f(y) \paren[\big]{ h'\lambda + h\lambda' + h'' }(y)} \label{eq:y'(tau) using exp(delta tau)},
\intertext{if the denominator is nonzero. Also note that}
	\ybar'(0) &= \frac{- \delta h'(y_0)}{ h'(y_0)\lambda(y_0) + \paren[\big]{ h\lambda }'(y_0)	} < 0
\end{align}
by \cref{cond:model parameters} as $h' > 0$, $(h \lambda)' > 0$ and $\lambda > 0$.
Hence, there exists a maximal $T_\infty \in (0,\infty]$ such that $\ybar'(\tau) < 0$ for $\tau \in [0,T_\infty)$, so $\ybar$ is bijective there.
Call $\tau(y):= \ybar^{-1}(y)$ its inverse and let ${y_\infty := \lim_{\tau \nearrow T_\infty} \ybar(\tau) < y_0}$.
By \eqref{eq: exp TTL of y}, \cref{eq:y'(tau) using exp(delta tau)} simplifies to
\begin{equation} \label{eq:y'(tau) simplified}
	\ybar'(\tau) = \frac{\delta \paren[\big]{h\lambda + h' + \delta}(y) h'(y)}{\paren[\big]{h'' - h' \lambda }(y) \paren[\big]{h\lambda + h' + \delta}(y) - \paren[\big]{h'\lambda + h \lambda' + h''}(y) h'(y)}
\end{equation}
for ${y = \ybar(\tau)}$.
By definition of $T_\infty$ and $y_\infty$, we see that $\ybar'(\tau)$ is negative on $[0,T_\infty)$ and $0$ at $y=y_\infty$.
Hence $h(y_\infty) \lambda(y_\infty) + h'(y_\infty) + \delta = 0$, which justifies the notation $y_\infty$ as in \cref{cond:model parameters}, according to which such a unique solution $y_\infty < y_0$ exists.
An ODE for $\theta(y)$ on $y \in (y_\infty, y_0]$ is obtained from \eqref{eq:diff thetabar} via
\begin{align}
	\theta'(y) &= \frac{\diff}{\diff y} \thetabar(\tau(y)) = \thetabar'(\tau(y))\tau'(y) \notag
= \paren[\big]{\ybar'(\tau(y)) - h(y)} \frac{1}{\ybar'(\tau(y))} = 1 - \frac{h(y)}{\ybar'(\tau(y))} \notag
\\		&= 1 - \frac{h(y)}{\delta h'(y)} \paren[\big]{h'' - \lambda h'}(y) + \frac{h(y) \paren[\big]{h \lambda + h' + \delta}'(y)}{\delta \paren[\big]{h \lambda + h' + \delta}(y)} \label{eq:theta' almost}
\end{align}
with $\theta(y_0)=0$.
We also note that \eqref{eq:theta' almost} equals \eqref{eq:theta' common}.

\subsection{Properties of the candidate for the free boundary}
\label{sec:prop of the candidate boundary}

To justify some assumptions in the analysis above, we verify here  previously presumed properties for the candidate boundary, especially bijectivity of $\theta : (y_\infty, y_0] \to [0,\infty)$.

\begin{figure}[ht]%
	\centering
	\includegraphics[width=0.66\textwidth]{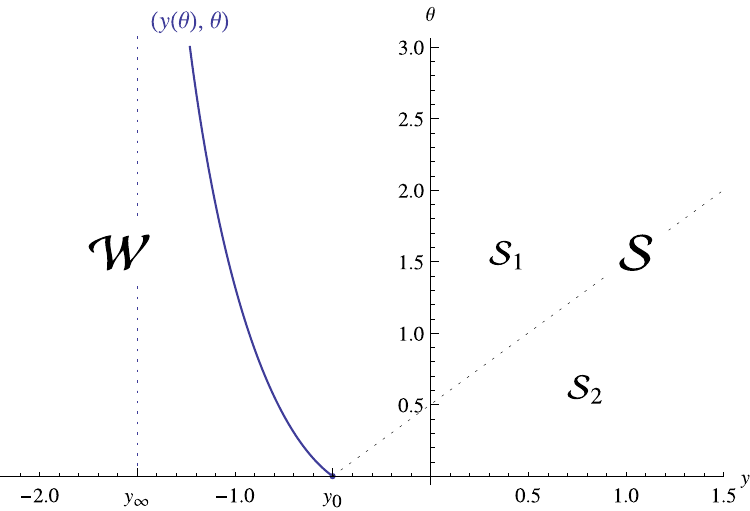}
	\caption{The division of the state space, for $\delta=0.5$, $h(y)=y$ and $\lambda(y) \equiv 1$.}
	\label{fig:wait-and-sell-region}
\end{figure}

\begin{lemma}\label{lemma: boundary}
	The function $\theta:(y_\infty, y_0] \to \RR$ defined in \eqref{eq:def of boundary} is a strictly decreasing $C^1$-function that maps bijectively $(y_\infty, y_0]$ to $[0, \infty)$ with $\theta(y_0) = 0$ and $\lim_{y \searrow y_{\infty}} \theta(y) = +\infty$.
\end{lemma}
\begin{proof}
	By \cref{cond:model parameters} we have that $h' > 0$ and $y\mapsto (\delta + h \lambda + h')(y)$ is strictly increasing, giving that the denominator in \eqref{eq:boundary derivative} is strictly positive when $y > y_\infty$.
	Thus, to verify that $\theta$ is decreasing it suffices to check that the numerator in \eqref{eq:boundary derivative} is negative.
	For this, we write the numerator as
	\[
		(\delta + h \lambda)(h')^2 + (\delta + h\lambda)^2 h' + h h'(h\lambda)' - h(\delta + h\lambda)h''.
	\]
	Note that $(\delta + h \lambda)(h')^2 + (\delta + h\lambda)^2 h' = (\delta + h\lambda)(\delta + h' + h\lambda)h' < 0$ for all $y \in (y_\infty, y_0)$, since $y_0 < 0$.
	Hence, $\theta'(y) < 0$ for $y\in (y_\infty, y_0)$ follows by 
	\begin{align*}
		h h'(h\lambda)' - h(\delta + h\lambda)h'' 
		&< h h'(h\lambda)' - h(\delta + h\lambda)h'' - h(\delta + h' + h\lambda)(h\lambda)'
	\\
		&= - h(\delta + h\lambda)(h \lambda + h')' < 0.
	\end{align*}
	It is clear that $\theta$ defined in \eqref{eq:def of boundary} is $C^1$.
	So it remains to verify $\lim_{y\searrow y_{\infty}} \theta(y) = +\infty$.
	Note that the arguments above actually show that the numerator of the integrand is bounded from above by a constant $c < 0$ when $x\in [y_{\infty}, y_0]$.
	Also, since the derivative of the denominator is bounded on $[y_\infty, y_0],$ we have by the mean value theorem 
	\[
		0\leq \delta \paren[\big]{h' (\delta +h\lambda h')}(x)\leq C (x-y_{\infty}),\quad x\in (y_{\infty}, y_0],
	\]
	for a finite constant $C>0$.
	Thus, we can estimate
	\[
		\theta(y) \geq \int_y^{y_0}\frac{-c}{C(x-y_{\infty})} \diff x  = \frac{-c}{C} \paren[\big]{\log(y_0 - y_\infty) - \log (y - y_\infty)}\quad \forall y\in(y_{\infty}, y_0],
	\]
	which converges to $+\infty$ as $y\searrow y_{\infty}$.
	This finishes the proof.
\end{proof}

\subsection{Constructing the value function and the optimal strategy} \label{subsection: Construction of V and the optimal strategy}
\label{subsec: constr soln}
The smooth pasting approach directly gives the value function $V$ along the boundary as
\begin{equation} \label{eq: V boundary}
	V(y(\theta),\theta) = \Vboundary(\theta) := \left. f(y) h(y) \frac{\delta + h(y)\lambda(y)}{\delta h'(y)} \right|_{y=y(\theta)}
\end{equation}
via \cref{eq:V in wait region,eq:C on the boundary}.
In the calculus of variations approach, we get~\eqref{eq: V boundary} as the solution to \eqref{task:maximize J(y)} after inserting \cref{eq: exp TTL of y}, doing a change of variables with \eqref{eq:diff thetabar} and applying \cref{lemma:integral for Vboundary}.
By \cref{eq:var ineq I}, we can extend $V$ into the wait region:
\begin{align}
	V(y,\theta) = V^{\scW}(y,\theta)&:= \Vboundary(\theta) \exp\paren[\bigg]{\int_{y(\theta)}^y \frac{-\delta}{h(x)} \diff x} \notag
\\		&= \paren[\bigg]{ \frac{f h \, (\delta + h\lambda)}{\delta h'} }\!(y(\theta)) \exp\paren[\bigg]{\int_y^{y(\theta)} \frac{\delta}{h(x)} \diff x} \label{eq:V wait}
\end{align}
for $(y,\theta) \in \widebar\scW$.
Using \cref{eq:var ineq III} we get $V$ inside ${\scS_1 := \scS \cap \{ (y,\theta) \mid y < y_0 + \theta \}}$ as follows.
For ${(y,\theta) \in \widebar\scS_1}$ let $\Delta:= \Delta(y,\theta)$ be the $\norm{\cdot}_\infty$-distance of~$(y,\theta)$ to the boundary in direction~$(-1,-1)$, i.e.
\begin{equation} \label{def:distance to boundary}
	\theta = \theta_* + \Delta, \qquad y = y(\theta_*) + \Delta, \qquad \Delta \ge 0.
\end{equation}
We then have for $y(\theta) \le y \le y_0 + \theta$, that
\begin{align}
	V(y,\theta) = V^{\scS_1}(y,\theta) &:= \Vboundary(\theta_*) + \int_0^\Delta f(y_1 + x) \diff x \label{eq: unsimplified V in difficult sell region}
\\		&= \paren[\bigg]{ \frac{f h \, (\delta + h\lambda)}{\delta h'} }\!(y - \Delta) + \int_{y-\Delta}^y f(x)\diff x. \label{eq: V in difficult sell region}
\end{align}
Similarly, with \cref{eq:var ineq boundary cond} we obtain $V$ in $\scS_2:= \scS \backslash \widebar \scS_1$, i.e.~for $y \ge y_0 + \theta$:
\begin{equation} \label{eq: V in simple sell region}
	V(y,\theta) = V^{\scS_2}(y,\theta) := \int_{y-\theta}^y f(x)\diff x.
\end{equation}
Since $\Vboundary(0)=0$, we can combine $V^{\scS_1}$ and $V^{\scS_2}$ by extending ${\Delta(y,\theta) := \theta}$ inside $\scS_2$.
So $\Delta:= \Delta(y,\theta)$ is the $\norm{\cdot}_\infty$-distance in direction $(-1,-1)$ of the point $(y,\theta) \in \scS$ to $\partial \scS$ and
\begin{equation} \label{eq: V in sell region}
	V(y,\theta) = V^{\scS}(y,\theta) := \Vboundary(\theta - \Delta) + \int_{y-\Delta}^y f(x)\diff x
\end{equation}
for all $(y,\theta) \in \scS$.
But note that $y(\theta - \Delta) = y - \Delta$ only holds in $\widebar \scS_1$, not in $\scS_2$.
After resuming the properties of $V$ in the next lemma (proved in \cref{appendix:proofs about value functions}), we can prove our main result \cref{thm: optimal strategy}.

\begin{lemma}\label{lemma: V continuously differentiable}
	The function $V: \RR \times [0,\infty) \to \RR$ with
	\[
		V(y,\theta) = \begin{cases} 
			\Vboundary(\theta-\Delta) + \int_{y-\Delta}^y f(x)\diff x, & \text{for $y \ge y(\theta)$,} 
		\\	\Vboundary(\theta) \cdot \exp\paren[\big]{\int_{y(\theta)}^y \frac{-\delta}{h(x)} \diff x}, & \text{for $y \le y(\theta)$,} 
		\end{cases}
	\]
	as defined by \cref{eq: V boundary,eq:V wait,eq: V in sell region} is in $C^{1}(\RR \times [0,\infty))$ and solves the free boundary problem \eqref{eq:var ineq I} -- \eqref{eq:var ineq boundary cond}.
\end{lemma}

\begin{proof}[Proof of \cref{thm: optimal strategy}]
	On admissibility of $\assetsProcess^\text{opt}$: Predictability of $\assetsProcess^\text{opt}$ is obvious by continuity of $y(\theta)$.
	In fact, $\assetsProcess^\text{opt}$ is deterministic because $Y_t$ is so.
	As noted in the proof of \cref{lemma: difficult sell region inequality}, the function ${y \mapsto \paren[\big]{f \cdot (h\lambda + h' + \delta) / h'}(y)}$ is increasing in ${(y_\infty, y_0]}$, so $\tau(y)$ and its inverse $\ybar(\tau)$ are decreasing, as is $\theta(y)$ by \cref{lemma: boundary}.
	This implies monotonic decrease of $\assetsProcess^\text{opt}$.
	Right continuity follows from the description of the 5 steps stated in the theorem.
	So $\assetsProcess^\text{opt} \in \admissibleSellStrategies{\assetsProcess_{0-}}$.
	
	On finite time to liquidation: By $h\lambda + h' + \delta > 0$ in $(y_\infty, y_0]$ and \cref{eq: exp TTL of y}, it takes $\tau(Y_s) < \infty$ time to liquidation if $Y_s = y(\assetsProcess_s)$, i.e.~along the boundary.
	This time only increases by some waiting time $s > 0$ in case $Y_{0-} < y(\assetsProcess_{0-})$ (step~\labelcref{strategy-step:wait}).
	But since $h(y_0) < 0$, we have $s < \infty$.
	
	On optimality: Note that $(Y_t, \assetsProcess_t) \in \brackets[\big]{\min \{y-\assetsProcess_{0-}, 0\}, \max\{0,y\}} \times \brackets[\big]{0, \assetsProcess_{0-}}$ for $y = Y_{0-}$ because $h(0)=0$ and $h'>0$.
	So $V(Y_t, \assetsProcess_t)$ is bounded by continuity of $V$ (\cref{lemma: V continuously differentiable} above).
	So the local martingale part of $G$ in \cref{eq: Ito for G} is a true martingale on every compact time interval because $M$ is a square-integrable martingale by assumption.
	By construction of $V$ and \cref{lemma: V continuously differentiable,lemma: wait region inequality,lemma: simple sell region inequality,lemma: difficult sell region inequality}, $G$ is a supermartingale with $G_0 - G_{0-} = \baseS_0 \int_0^{\Delta \assetsProcess_0} \paren[\big]{ V_y + V_\theta - f }(Y_{0-} + x, \assetsProcess_{0-} + x) \diff x \le 0$ for every strategy and a true martingale with $G_{0-} = G_0$ for $\assetsProcess^\text{opt}$.
	So \cref{prop: supermatringale suffices} applies.
\end{proof}

\FloatBarrier

\begin{figure}[!ht]%
	\centering
	\begin{subfigure}{0.45\textwidth}
		\centering
		\includegraphics[width=\textwidth]{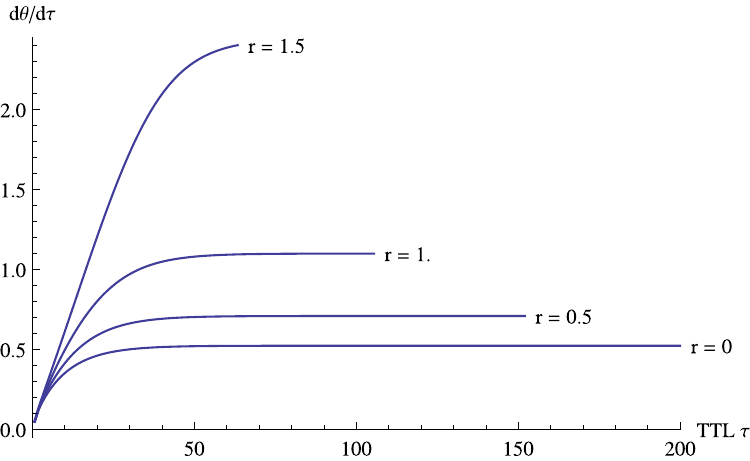}
		\caption{Fixed $\beta=1$, $\delta=0.1$, varying $r$.}
		\label{fig:liquidation-rate-different-r}
	\end{subfigure}
	~
	\begin{subfigure}{0.45\textwidth}
		\centering
		\includegraphics[width=\textwidth]{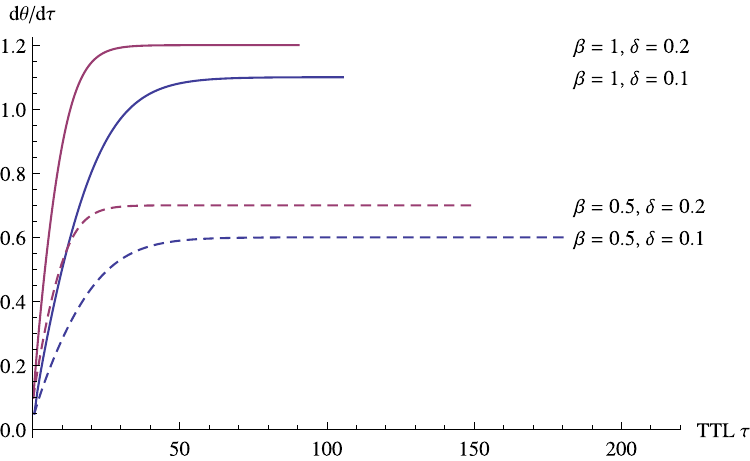}
		\caption{Fixed $r=1$, varying $\beta$ and $\delta$.}
		\label{fig:liquidation-rate-different-beta-and-delta}
	\end{subfigure}
	\caption{Liquidation rate (after initial block trade) in Ex.\ \ref{ex:simple-lob}. Lines end at $\protect\thetabar(\tau) = 100$.}
	\label{fig:liquidation-rate}
\end{figure}
\begin{figure}[!ht]%
	\centering
	\begin{subfigure}{0.45\textwidth}
		\centering
		\includegraphics[width=\textwidth]{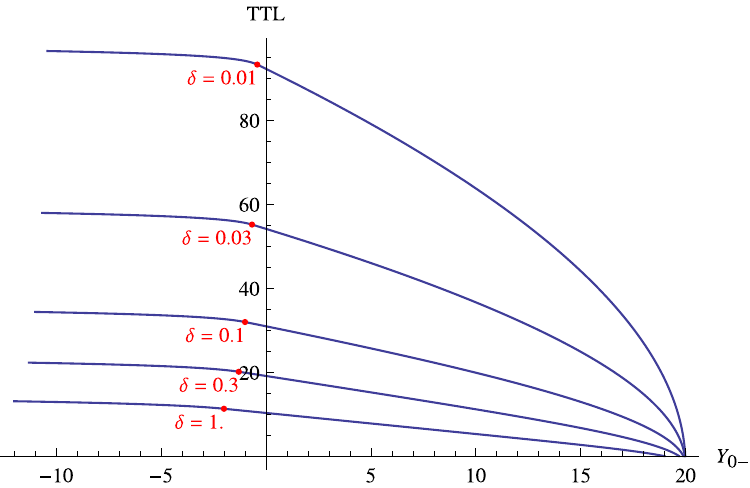}
		\caption{Fixed $\beta=1$, $r=1$, varying $\delta$.}
		\label{fig:ttl-over-y0-different-delta}
	\end{subfigure}
	~
	\begin{subfigure}{0.45\textwidth}
		\centering
		\includegraphics[width=\textwidth]{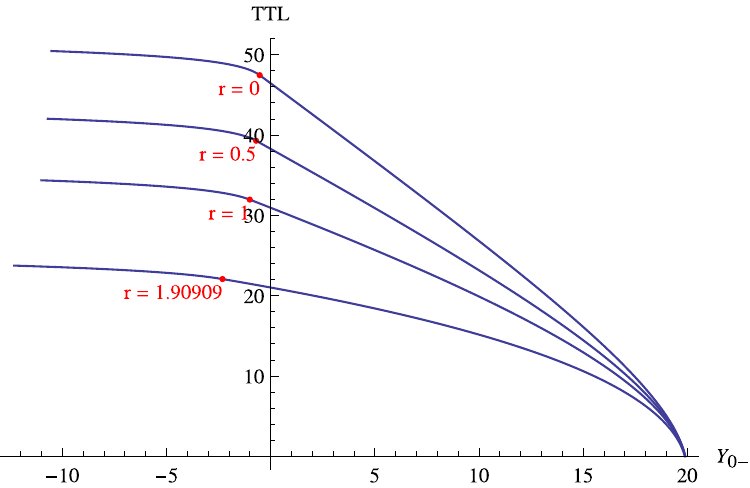}
		\caption{Fixed $\beta=1$, $\delta=0.1$, varying $r$.}
		\label{fig:ttl-over-y0-different-r}
	\end{subfigure}
	\caption{Dependence of TTL on $Y_{0-}$ for $\assetsProcess_{0-}=20$ in \cref{ex:simple-lob}. %
	A red point marks $(y(\assetsProcess_0), \tau(\assetsProcess_0))$, where continuous trading begins.}
	\label{fig:ttl-over-y0}
\end{figure}

\begin{example}\label{ex:simple-lob}
	Recall \cref{{ex:order book densities}} and let $h(y) = \beta y$ for $\beta > 0$.
	Then $(h\lambda)'>0$ and
	\[
		y_0 = \frac{-c \delta}{\beta + (1-r)\delta} \quad\text{and}\quad y_\infty = \frac{-c(\beta+\delta)}{\beta + (1-r)(\beta + \delta)}.
	\]
	As can be seen from the proofs, $\lambda$ and $h$ are only needed at possible values of $Y_t$. 
	Hence \cref{cond:model parameters} effectively restricts the state space $\overbar \scW \cup \overbar \scS$ to $c + (1-r)y > 0$. %
	We only have to check this for $Y_{0-}$, $y_0$ and $y_\infty$. %
	Note that the special case $Y_{0-} = 0$ already does so.
	Now $y_\infty$ and $y_0$ lie in the required range with $y_\infty < y_0 < 0$ if $r \in \bigl[0, 1 + \beta/(\beta + \delta)\bigr)$.
	For $r\ne 1$ and $y \in (y_\infty, y_0]$ we get%
	\[
		\hspace{-1em}
		\theta(y) = \frac{\beta y + \delta A(y)}{\delta (1-r)}
			- \frac{\beta c B}{\delta C (1-r)^2} \log\paren[\Big]{ \frac{A(y) B}{\beta c} }
			+ \frac{\beta c (\beta + \delta)}{\delta C} \log\paren[\Big]{ \frac{\beta A(y)}{\beta y + (\beta + \delta)A(y)} }\,,
	\]
	with $A(y):= c+ (1-r)y$, $B:= \beta + (1-r)\delta$ and $C:= \beta + (1-r)(\beta + \delta)$,
	whereas
	\[
		\theta(y) = \frac{(\beta y + \delta c)(\beta y + (2\beta+\delta)c)}{2\beta\delta c} - c \frac{\beta + \delta}{\delta} \log\paren[\Big]{\frac{\beta y + (\beta+\delta)c}{\beta c}}
		\quad\text{if $r=1$.}
	\]
	Time to liquidation (TTL) \emph{along the boundary} is a function of $\delta$, $\beta$, $r$ and $y/c$, namely
	\begin{align*}
		\tau(y) &= \begin{cases}
			\displaystyle
			-\frac{1}{\delta}\log\paren[\bigg]{ \paren[\Big]{\frac{1 + (1-r) y/c }{1 + (1-r) \delta/\beta}}^{1/(1-r)} \paren[\Big]{ \frac{ y/c }{1 + (1-r) y/c} + \frac{\beta + \delta}{\beta} } }
			&\text{if $r\ne 1$,}
		\\[1em]
			\displaystyle
			-\frac{y}{\delta c} - \frac{1}{\beta} - \frac{1}{\delta} \log\paren[\Big]{ \frac{y}{c} + 1 + \frac{\delta}{\beta} }
			&\text{if $r=1$.}
		\end{cases}
	\end{align*}
	For $r=1$ and $\delta \to \infty$, we have $y_0 \to -\infty$, so the overall TTL tends to and ultimately equals $0$, a single block sale, $\assetsProcess_0=0$, being optimal for sufficiently large $\delta$.
	Using the product logarithm $W:=(x \mapsto x e^x)^{-1}$, impact and asset position for $r=1$, $\tau\ge 0$ are
	\begin{align*}
		\ybar(\tau) &= c W\paren[\big]{ e^{1-\delta\tau} } - c\frac{\beta + \delta}{\beta}
			\qquad\text{and}
	\\
		\thetabar(\tau) &= \frac{\beta c}{2\delta} \paren[\Big]{ W\paren[\big]{ e^{1-\delta\tau} }^2 - 1 } - c\frac{\beta + \delta}{\delta} \log W\paren[\big]{ e^{1-\delta \tau} }
	\end{align*}
	along the free boundary.
	The rate $\!\diff \assetsProcess_t/\diff t = -\thetabar'(T-t)$ becomes asymptotically constant for TTL $\tau = T-t \to \infty$ and decreases for $\tau \to 0$, cf.\ \cref{fig:liquidation-rate}. 
	The respective limits are
	\begin{align*}
		\lim_{\tau\to \infty} \thetabar'(\tau) = -h(y_\infty) = \frac{\beta c (\beta + \delta)}{\beta + (1-r)(\beta + \delta)}\quad& \text{ and  }\quad 
 \thetabar'(0) = \frac{\delta \beta c}{2\beta + (1-r)\delta}\,.
	\end{align*}
\end{example}

\begin{remark}\label{problem:optimalaquisition}
	How to optimally acquire an asset position, minimizing the expected costs, is the natural counterpart to the previous liquidation problem; cf.\ \cite{PredoiuShaikhetShreve11}.
	To this end, if we represent the admissible strategies  by increasing c\`{a}dl\`{a}g processes $\assetsProcess$ starting at 0 (describing the cumulative number of shares purchased over time), then the discounted costs (negative proceeds) of an admissible (purchase) strategy $\assetsProcess$ takes the form 
	\begin{equation}
		\int_0^{\infty}e^{ \eta t}f(Y_{t-})M_t \diff \assetsProcess^c_t + \sum_{\substack{t \geq 0 \\ \Delta \assetsProcess_t \neq 0}} e^{ \eta t} M_t \int_0^{\Delta \assetsProcess_t} f(Y_{t-} + x) \diff x,
	\end{equation}
	with discounted unaffected price process $e^{-\gamma t}\baseS_t = e^{\eta t}M_t$ for $\eta := \mu - \gamma = -\delta$.
	To have a well-posed minimization problem for infinite horizon, one needs to assume 
that the price  process increases in expectation, i.e.\  $\eta > 0$, and thus the trader aims to buy an asset with rising (in expectation) price.
	
	In this case, the value function %
	of the optimization problem
	will be described by the variational inequality
	\(
		\min \{f + V_y - V_\theta\ ,\ \eta V - h V_y\} = 0.
	\)
	An approach as taken previously for the optimal liquidation problem  permits again to construct the classical solution to this free boundary problem explicitly.
	 Thereby, the state space is divided into a wait region and a buy region by the free boundary, that is described by 
	\begin{equation}\label{eq:boundary purchase}
		\theta'(y) = -1 + \frac{h(y)\lambda(y)}{\eta} - \frac{h(y)h''(y)}{\eta h'(y)} + \frac{h(y)(h\lambda + h' - \eta)'(y)}{\eta(h\lambda + h' - \eta)(y)},\qquad y\geq y_0,
	\end{equation}
	with initial condition $\theta(y_0) = 0$, where $y_0$ is the unique root of $h(y)\lambda(y) = \eta$ (similar to \eqref{eq:theta' common} from the optimal liquidation problem).
	Verification of %
	optimality
	will go though under the assumptions $\eta > 0$, $f \in C^2$ with $f(0) = 1$, $\lambda(y):= f'(y)/f(y) > 0$, resilience $h \in C^2$ with $h(0) = 0$, $h' > 0$, the technical condition $(h')^2 > h h''$ and such that $(h\lambda)' > 0$ and $(h\lambda + h')' > 0$.
	Note that apart from $\eta > 0$ and $(h')^2 > h h''$, these match \cref{cond:model parameters}.
	Examples satisfying $(h')^2 > h h''$ are $h(y) = \beta y$ and $h(y) = \alpha \arctan(\beta y)$ for $\alpha,\beta > 0$.
	
	It may be interesting to note that the boundary defined by \eqref{eq:boundary purchase} does not have a vertical asymptote, because such an asymptote could only occur at a root $y_\infty$ of the denominator $h\lambda + h' - \eta$, but $y_\infty < y_0$ and $\theta'(y_0) = \bp{h(h\lambda)'/(\eta h')}(y_0) > 0$.
	The technical condition $(h')^2 > h h''$  guarantees that the boundary is strictly increasing for all $y \ge y_0$.
\end{remark}

\section{Optimal~liquidation with non-monotone strategies}
\label{sect:intermediate buy orders}

In this section, we solve under \cref{cond:model parameters} the optimal liquidation problem when the admissible liquidation strategies allow for intermediate buying. 
To focus again on transient price impact and explicit analytical results, we keep other model aspects simple and consider the problem in a two-sided order book model with zero bid-ask spread. This is an idealization of the predominant one-tick-spread 
that is observed for common relatively liquid risky assets \cite{ContLarrard13}. See \cref{rem:nonzerospread} though. 
We show that the optimal trading strategy is monotone when $Y_{0-}$ is not too small (see \cref{rem:round-trips-in-depressed-market}).
More precisely, the two-dimensional state space decomposes into a buy region and a sell region with a non-constant interface, that coincides with the free boundary constructed in \cref{sec:Free boundary problem}.

In previous sections, we considered pure selling strategies and  specified the model for such, i.e.\ in the sense of \cref{sect:LOB} we specified only the bid side of the LOB.
Now, we extend the model to allow for buying as well.
In this case, a large investor's trading strategy may be described by a pair of increasing c\`{a}dl\`{a}g processes $(A^+, A^-)$ with $A^{\pm}_{0-} = 0$, where $A_t^+$ (resp.\ $A_t^-$)  describes the cumulative number of assets sold (resp.\ bought)  up to time $t$.
Her risky asset position is $\assetsProcess_t = \assetsProcess_{0-} - (A^+_t - A^-_t)$ at time $t \ge 0$.
We assume that the price impact process $Y = Y^\assetsProcess$ is given by \eqref{eq:deterministic Y_t dynamics} with $\assetsProcess = \assetsProcess_{0-} - (A^+ - A^-)$, and that the best bid and ask prices evolve according to the same process $S = f(Y^\assetsProcess)\baseS$, i.e.~the bid-ask spread is taken as zero.
The proceeds from executing a market buy order at time $t$ of size $\Delta A^{-}_t > 0$ are given again by \eqref{eq:block sale proceeds} with $\Delta \assetsProcess_t = \Delta A^-_t$.
Proceeds being negative means that the trader pays for acquired assets.
Thus, the $\gamma$-discounted (cumulative) proceeds from
trading strategy $(A^+, A^-)$ over time period $[0,T]$ are
\begin{equation}\label{eq:proceeds fv strategy}
	L_T = -\int_{0}^T e^{-\gamma t} f(Y_{t}) \baseS_t \diff \assetsProcess_t^c - \sum_{\substack{\Delta \assetsProcess_t \neq 0 \\ t\leq T}} e^{-\gamma t}\baseS_t \int_0^{\Delta \assetsProcess_t} f(Y_{t-} + x) \diff x\,.
\end{equation}
For finite variation strategies $\assetsProcess$ 
the sum in \eqref{eq:proceeds fv strategy} converges absolutely, cf.\ \cref{rmk:convergence of block proceeds}.

We consider the optimization problem over the set of admissible trading strategies
\begin{align}
	\hspace{-1.5em}
	\admissibleFiniteVariationStrategies{\assetsProcess_{0-}} := \bigl\{ \squeeze[1]{\assetsProcess = \assetsProcess_{0-} - (A^+ - A^-)} \mid {}&\text{$A^{\pm}$ is increasing, càdlàg, predictable, \rlap{bounded,}} 
	\nonumber
\\ \label{eq:admfvstrat}
			&\text{ with $A^{\pm}_{0-}= 0$ and $\assetsProcess_t \ge 0$ for $t \ge 0$} \bigr\},
\end{align} 
where $A =  A^+ - A^-$ denotes the minimal decomposition for a process $A$ of finite (here even bounded) variation; the last condition means that short-selling is not allowed.%

For an admissible strategy $\assetsProcess \in \admissibleFiniteVariationStrategies{\assetsProcess_{0-}}$, $L_T(y;\assetsProcess)$ as defined in \eqref{def:liquidation proceeds}, but extended to general bounded variation strategies by \eqref{eq:proceeds fv strategy}, describes the proceeds from implementing $\assetsProcess$ on the time interval $[0, T]$. 
These proceeds are a.s.~finite for every $T\geq 0$, see \cref{rmk:convergence of block proceeds}.
We now show that  $\lim_{T\rightarrow\infty} L_T(y;\assetsProcess)$ exists in $L^1$. Let $L(y;\assetsProcess) = L^+(y;\assetsProcess) - L^-(y,\assetsProcess)$ be the minimal decomposition of the process $L(y;\assetsProcess)$ (having finite variation), and let $L(y;\assetsProcess^\pm)$ be the proceeds process from a monotone strategy $\assetsProcess^{\pm} = \assetsProcess_{0-} \mp A^{\pm}$. 
We have $L^\pm_T(y;\assetsProcess) \leq L_T(y;\assetsProcess^\pm)$ for every $T\geq 0$ because $Y^{\assetsProcess^-} \geq Y^\assetsProcess \geq Y^{\assetsProcess^+}$.  
Moreover, since $A^+ + A^{-} \leq C$ for some constant $C$ we conclude from the solution of the optimization problem with monotone strategies (\cref{thm: optimal strategy}) that $L^\pm_T(y;\assetsProcess) \leq L_\infty(y;\assetsProcess^\pm) \in L^1$ for every $T\geq 0.$ 
By dominated convergence one gets that $L^{\pm}_T(y;\assetsProcess) \to L^{\pm}_\infty(y;\assetsProcess)$ in $L^1$ and a.s.\ for $T \to \infty$.
In particular, $\lim_{T \to \infty} L_T(y;\assetsProcess) = L^+_\infty(y;\assetsProcess) - L^-_\infty(y;\assetsProcess) =: L_{\infty}(y;\assetsProcess)$ exists in $L^1$.
So, the gain functional $J(y;\assetsProcess)$ for the optimal liquidation problem with possible intermediate buying,
\begin{equation}\label{task:optimal liquidation with intermediate buys}
	\max_{\assetsProcess \in \admissibleFiniteVariationStrategies{\assetsProcess_{0-}}} J(y;\assetsProcess)\qquad \text{for}\qquad J(y;\assetsProcess) := \EE[L_\infty(y;\assetsProcess)]\,,
\end{equation}
is well defined. 
By arguments as in \cref{sec:Problem} (cf.~\cref{prop: supermatringale suffices} and \eqref{eq: Ito for G}) one sees that in this case it suffices to find a classical solution to the following problem
\begin{align}
	V_y + V_\theta &= f \qquad\text{on }\RR\times[0,\infty)\label{eq:ineq when interm buying I},
\\	-\delta V - h(y)V_y &\leq 0 \qquad\text{on }\RR\times[0,\infty) \label{eq:ineq when interm buying II},
\end{align}
with suitable boundary conditions, ensuring that a classical solution exists and that the (super-)martingale properties from \cref{prop: supermatringale suffices} extend to $[0-, T]$, cf.\ \cref{rem:supermartprop}.
The optimal liquidation strategy then can be described by a sell region and a buy region, divided by a boundary.

The sell region turns out to be the same as for the problem without intermediate buying in \cref{sec:Problem}, i.e.~the region $\scS$, while the wait region $\scW$ there becomes a buy region
\(
	\buyReg := \RR\times [0,\infty) \setminus \widebar{\scS}
\) 
here.
Similarly to \cref{subsection: Construction of V and the optimal strategy}, we extend the definition of $\Delta(y,\theta)$ to $\buyReg$.
For $(y,\theta)\in \RR\times [0,\infty)$, let $\Delta(y,\theta)$ be the signed $\norm{\cdot}_1$ distance in direction $(-1,-1)$ of the point $(y,\theta)$ to the boundary $\partial \scS = \{(y(\theta), \theta) \mid \theta\geq 0\} \cup \{(y,0) \mid y\geq y_0\}$, i.e.~$(y-\Delta, \theta-\Delta) \in \partial S$.
Recall the definition of $V^{\scS}$ in \eqref{eq: V in sell region} and let
\[
	V^{\buyReg}(y,\theta) := \Vboundary(\theta - \Delta(y,\theta))- \int_{y}^{y-\Delta(y,\theta)} f(x)\diff x,\quad \text{for }(y,\theta)\in \buyReg.
\]
The discussion so far suggests that the following function would be a classical solution to the problem \eqref{eq:ineq when interm buying I} -- \eqref{eq:ineq when interm buying II} describing the value function of the optimization problem \eqref{task:optimal liquidation with intermediate buys}:
\begin{equation}\label{eq:value fn interm buy}
	\Vbuysell(y, \theta) := 
		\begin{cases}
			V^\scS(y,\theta), & \hbox{if }(y,\theta)\in \widebar{\scS},
		\\	V^{\buyReg}(y,\theta),& \hbox{if }(y,\theta)\in \buyReg,
		\end{cases}
\end{equation}
up to the multiplicative constant $\baseS_0$.
Note that both cases in \eqref{eq:value fn interm buy} can be combined to 
\[
	\Vbuysell(y,\theta) = \Vboundary(\theta - \Delta(y,\theta))+ \int_{y-\Delta(y,\theta)}^{y} f(x)\diff x,\quad \text{for all }(y,\theta).
\]
The next theorem proves the conjectures already stated in this section for solving the optimal liquidation problem with possible intermediate buying.

\begin{theorem}\label{thm:optimal liquidation with intermediate buy orders}
	Let the model parameters $h$, $\lambda$, $\delta$ satisfy \cref{cond:model parameters}.
	The function $\Vbuysell$ is in $C^{1}(\RR\times [0,\infty))$ and solves \eqref{eq:ineq when interm buying I} and \eqref{eq:ineq when interm buying II}.
	The value function of the optimization problem \eqref{task:optimal liquidation with intermediate buys} is given by $\baseS_0\cdot  \Vbuysell$.
	Moreover, for given number of shares $\assetsProcess_{0-}\geq 0$ to liquidate and initial state of the market impact process $Y_{0-} = y$, the unique optimal strategy $\assetsProcess^{\text{opt}}$ is given by $\assetsProcess^{\text{opt}}_{0-} = 0$ and:
	
	\begin{enumerate}
		\item 
			If $(y, \assetsProcess_{0-})\in \widebar{\scS}$, $\assetsProcess^{\text{opt}}$ is the liquidation strategy for $\assetsProcess_{0-}$ shares and impact process starting at $y$ as described in \cref{thm: optimal strategy}.
			
		\item 
			If $(y,\assetsProcess_{0-})\in \buyReg$, $\assetsProcess^{\text{opt}}$ consists of an initial buy order of $\abs{\Delta(y, \assetsProcess_{0-})}$ shares (so that the state process $(Y, \assetsProcess)$ jumps at time $0$  to the boundary between $\buyReg$ and $\scS$) and then trading according to the liquidation strategy for $\assetsProcess_{0-} + \abs{\Delta(y, \assetsProcess_{0-})}$ shares and impact process starting at $y + \abs{\Delta(y, \assetsProcess_{0-})}$ as described in \cref{thm: optimal strategy}.
	\end{enumerate}
\end{theorem}
\noindent
The proof of \cref{thm:optimal liquidation with intermediate buy orders} is given in \cref{appendix:proofs about value functions}.
By continuity arguments, one could show that the optimal strategy of \cref{thm:optimal liquidation with intermediate buy orders} is even optimal in a set of bounded semimartingale strategies (to which the definition of proceeds can be extended continuously in certain topologies on the c\`{a}dl\`{a}g space, see \cite[Section~5.2]{BechererBilarevFrentrup-model-properties}).

\begin{remark}
\label{rmk:transient impact is essential}
As already noted in \cite[Prop.~3.5(III)]{GuoZervos13}, a multiplicative order book with permanent instead of transient impact, i.e.~$h \equiv 0$, leads to a trivial optimal control with complete initial liquidation at time $0$, in absence of transaction costs.
This can also be seen directly as follows. 
If $h \equiv 0$ we have $Y^\assetsProcess_t = Y_{0-} - \assetsProcess_{0-} + \assetsProcess_t$ and proceeds \eqref{def:liquidation proceeds} may be written as
\begin{align}
	\hspace{-2em}
	\MoveEqLeft
	\label{eq: permanent impact proceeds}
	L_T(\assetsProcess) 
		= \int_0^T F(Y^\assetsProcess_t) \diff \bp{ e^{-\delta t} M_t }_t 
			- \bp{ e^{-\delta T} M_T F(Y^\assetsProcess_T) - M_0 F(Y_{0-}) }
\\ \notag
		&= -\delta \!\int_0^T \!\! e^{-\delta t} M_t F(Y^\assetsProcess_t) \diff t
			+ \int_0^T  \!\! e^{-\delta t} F(Y^\assetsProcess_t) \diff M_t
			- \bp{ e^{-\delta T} M_T F(Y^\assetsProcess_T) - M_0 F(Y_{0-}) }
\end{align}
with the antiderivative $F(y):= \int_{-\infty}^y f(x) \diff x \ge 0$ of $f$, assuming $F(0) < \infty$.
So we get for any two strategies $\assetsProcess$ and $\hat\assetsProcess$ with $\assetsProcess_t \ge \hat\assetsProcess_t$ for all $t \ge 0$, that $\EE[L_T(\assetsProcess)] \le \EE[L_T(\hat\assetsProcess)]$.
Thus it is optimal to liquidate all assets at time $0$, because $\hat\assetsProcess_t := 0 \le \assetsProcess_t$ for all $t \ge 0$ and $\assetsProcess \in \admissibleFiniteVariationStrategies{\assetsProcess_{0-}}$.
\Cref{eq: permanent impact proceeds} moreover shows that in the case of no drift ($\delta=0$) and permanent impact, \emph{every} strategy that liquidates until time $T$ is optimal.
This was already observed in \cite[comment before Prop.~3.5]{GuoZervos13} and shows a remarkable difference of effects from permanent and transient impact; cf.\ also \cref{rmk: finite time horizon; no drift; convexity}.
\end{remark}

\begin{remark} \label{rem:round-trips-in-depressed-market}
	The results show that when the initial level of market impact is sufficiently small, i.e.\ $Y_{0-}< y_0$, so that the market price is sufficiently depressed and has a strong upwards trend by~\eqref{eq:deterministic Y_t dynamics}, then the optimal liquidation strategy may comprise an initial block buy, followed by continuous selling of the risky asset position.
	In this sense our model admits \emph{transaction-triggered price manipulation} in the spirit of \cite[Definition~1]{AlfonsiSchiedSlynko12} for sufficiently small $Y_{0-}< y_0$. Let us note that  \cite[p.742]{LorenzSchied13} emphasize the particular relevance of the martingale case (zero drift) when analyzing (non)existence of price manipulation strategies, and that
 it seems natural to buy an asset whose price tends to rise.
The case $Y_{0-}< 0$ could be considered as adding an exogenous but non-transaction triggered upward component to the drift.
In any case, buying could only occur at initial time $t=0$ and afterwards the optimal strategy is just selling.  
	Nonetheless, for typical choices of the unperturbed price process $\baseS$ (e.g.~exponential Brownian motion)
	one can show that our model does not offer arbitrage opportunities (in the usual sense) for the large trader, 
and so strategies, whose expected proceeds are strictly positive, have to admit negative proceeds (i.e.\ losses)  with positive probability, see \cite[Section~4]{BechererBilarevFrentrup-model-properties}.

	On the other hand, if the level of market impact is not overly depressed, i.e.~$Y_{0-} \geq y_0$, then an optimal liquidation strategy will never involve intermediate buying.
	This includes in particular the case of a neutral initial impact $Y_{0-} = 0$ (as in \cite{PredoiuShaikhetShreve11}), or of an only mildly depressed initial impact $Y_{0-} \in [ y_0,\infty)$.
	Monotonicity of the optimal strategy would extend to cases with non-zero bid-ask spread, as explained below. %
\end{remark}

\begin{remark}[On non-zero bid-ask spread]\label{rem:nonzerospread}
	The results in this section also have implications for models with non-zero bid-ask spread.
	Indeed, if the initial market impact is not too small ($Y_{0-} \geq y_0$) and the LOB bid side is described as in our model, the optimal liquidation strategy in a model with non-zero bid-ask spread would still be monotone (so  relate only to the LOB bid side) and would be described by \cref{thm:optimal liquidation with intermediate buy orders}, since
	\[
		\sup_{\assetsProcess \in \admissibleSellStrategies{\assetsProcess_{0-}}} J(Y_{0-};\assetsProcess) 
		= \sup_{\assetsProcess \in \admissibleFiniteVariationStrategies{\assetsProcess_{0-}}} J(Y_{0-};\assetsProcess) 
		\geq \sup_{\assetsProcess \in \admissibleFiniteVariationStrategies{\assetsProcess_{0-}}} J^{\text{spr}}(Y_{0-};\assetsProcess),
	\]
	with $J^{\text{spr}}(Y_{0-};\assetsProcess)$ denoting the cost functional for the non-zero spread model, as $J(Y_{0-}, \cdot)$ and $J^{\text{spr}}(Y_{0-}, \cdot)$ coincide on $\admissibleSellStrategies{\assetsProcess_{0-}}$  and the inequality is due to the spread.%
\end{remark}

\begin{example}[Comparing multiplicative and additive impact]\label{ex:lorenz-schied}

\newcommand{\additiveMartingalePart}{N}
\newcommand{\additiveFiniteVariationPart}{K}

We want to highlight some differences between the optimal liquidation strategies for our model in comparison to the additive transient impact model of Lorenz and Schied~\cite{LorenzSchied13}, which generalizes the continuous time model as in \cite{ObizhaevaWang13}  by permitting non-zero drift for the unaffected price process. 
We will give a simple specification for both models below, which we will call the LS- and the mLOB-model.
With geometric Brownian motion $M_t := \scE(\sigma W)_t$,  Brownian motion $W$ and $\sigma>0$,
we take the unaffected price for both models to be given as in the standard Black-Scholes model by
\begin{equation}\label{eq:BSmodelfor compareLS}
	S^0_t = \baseS_t := \baseS_0 \scE(\mu t + \sigma W)_t = \baseS_0 + \additiveMartingalePart_t + \additiveFiniteVariationPart_t\quad \text{with} \quad \baseS_0  \in (0,\infty),
\end{equation}
with martingale part $N_t:= \int_0^t \sigma \baseS_s \diff W_s$ and finite variation part $\additiveFiniteVariationPart_t:= \int_0^t \mu \baseS_s \diff s$.

For bounded semimartingale strategies $X$ on $[0,T]$ with $X_{0-} = x$ and $X_t = 0$ for $t \ge T$, \cite{LorenzSchied13} define the price
at which trading occurs by
\(
	S^X_t := S^0_t + \eta E^X_{t-} \,,
\)
where $E^X_t := e^{-\rho t} \int_{[0,t]} e^{\rho s} \diff X_s$ is the volume impact process in a block-shaped LOB of height $1/\eta\in (0,\infty)$.
Note that $\diff E^X_t = -\rho E^X_t \diff t + \diff X_t$\,, and the same ODE is adhered by $Y^\assetsProcess_t$ and $\Theta$ by \eqref{eq:deterministic Y_t dynamics} for the resilience function  $h(y):= \rho y$.
We let $Y_{0-}:=0$ to have $Y=E^X$.

For the comparison, we still have to specify a multiplicative order book with similar features as the additive one from the LS-model.
Both order books should admit infinite market depth (LOB volume) for sell and for buy orders; and the prices should initially be similar for small volume impact $y$, i.e.~$\baseS_0 + \eta y \approx \baseS_0 f(y)$.
\Cref{ex:order book densities} then suggests as a simple specification $f(y) = e^{y/c}$
for the mLOB-model;
with $c:= \baseS_0 / \eta$ it further satisfies the requirement of similar prices up to first order.
Without loss of generality let $\eta=1$.
In the LS-model, the liquidation costs to be minimized in expectation are given by \cite[Lemma~2.5]{LorenzSchied13} as
\[
	\scC(X):= \int_{[0,T]} S^0_{t-} \diff X_t + [ S^0, X ]_T + \int_{[0,T]} E^X_{t-} \diff X_t + \frac{1}{2} [X]_T\,.
\]
According to \cite[Theorem~2.6]{LorenzSchied13}, the optimal semimartingale strategy $X$ with $X_{0-} = x$, minimizing $\EE[\scC(X)]$, is of the form
\begin{align*}
	X_t &= \frac{x(1+\rho(T-t)) - \frac{1}{2}(1+\rho t)Z_0}{2+\rho T} 
			- \frac{1}{2}\int_{(0,t]} \varphi(s) \diff Z_s 
			+ \frac{1}{2\rho} \additiveFiniteVariationPart'_t 
	\\	&\qquad- \rho\int_0^t \paren[\bigg]{ \frac{1}{2} \int_{(0,s]} \varphi(r) \diff Z_r + \frac{1}{2} \additiveFiniteVariationPart_s } \diff s\,,\quad t \in [0,T),
\end{align*}
 with $\varphi(t) = (2 + \rho(T-t))^{-1}$, derivative $\additiveFiniteVariationPart'_t := \diff \additiveFiniteVariationPart_t / \!\diff t = \mu \baseS_t$ and with $Z_t$ being equal to
\(
 \EE\brackets[\Big]{ \additiveFiniteVariationPart_T + \rho\int_0^T \additiveFiniteVariationPart_s \diff s \Bigm| \scF_t }.
\)
For unaffected price dynamics of Black-Scholes  type, this yields
$Z_0 = \paren[\big]{ (1 - e^{\mu T})(1+\frac{\rho}{\mu}) + \rho T } \baseS_0$ and
\(
	\diff Z_t = \paren[\big]{ \paren[\big]{ 1 - e^{\mu (T-t)} } \paren[\big]{1 + \frac{\rho}{\mu}} + \rho\, (T-t)  } \sigma \baseS_t \diff W_t\,.
\)
In particular, short-selling may occur and liquidation ends at time $T$ with a final block sale.
For the chosen price dynamics \eqref{eq:BSmodelfor compareLS}, the optimal liquidation strategy $X$ in the LS-model is a non-deterministic adapted semimartingale. 
As noted in \cite{LorenzSchied13}, it is not of finite variation.
In contrast, cf.\ \cref{rmk:optimal strategy is deterministic}, the optimal strategy from \cref{thm:optimal liquidation with intermediate buy orders} in our mLOB-model is deterministic and of bounded variation. 
As noted there, by continuity arguments our optimal strategy could be shown to be also optimal within a larger class of  bounded semimartingale strategies.
However, note that optimization in the mLOB-model is over a smaller set of strategies without short-selling.

If the parameters $\mu$, $\rho$, $\baseS_0$ and $\assetsProcess_{0-}$ are such that our optimal strategy for the infinite time horizon problem liquidates until the given time $T$, then it is clearly also optimal among all liquidation-strategies on $[0,T]$. 
Otherwise (if $T$ is too small), the \enquote{short-time-liquidation} problem in the case of non-zero drift $\mu$ in our model is still open, cf.\ \cref{rmk: finite time horizon; no drift; convexity}.
By \cref{ex:simple-lob}, for every $T$, $\rho$, $\baseS_0$ and $\assetsProcess_{0-}$, there exists some $\delta = \gamma - \mu$ such that liquidation occurs until time $T$.
\begin{figure}
	\centering
	\begin{overpic}[width=0.8\textwidth]{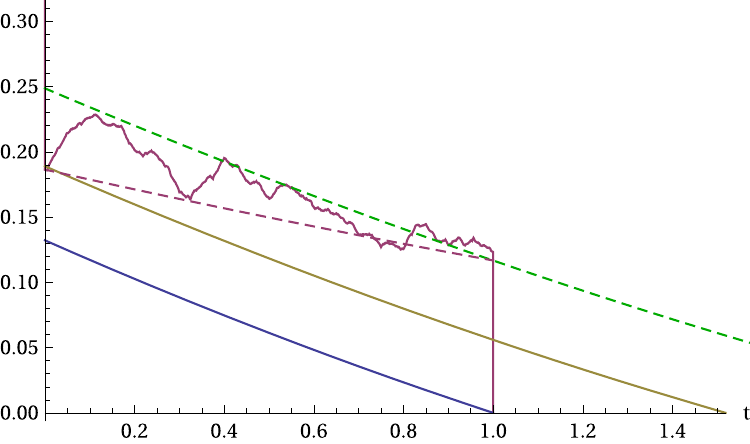}
		\put (7,60) {$\assetsProcess_t, X_t$}
	\end{overpic}
	\caption{
		Optimal risky asset position $X$, resp.\ $\assetsProcess$, over time with $\protect\baseS_0 = 1$, initial position $\protect\assetsProcess_{0-} = X_{0-} = 1$, $\mu = -0.2$, $\sigma = 0.5$, $\rho = 0.3$, in different models. 
		Red is one realization $X_\cdot(\omega)$ of the LS-model for time horizon $T=1$ with a red dashed line for $\EE[X_t]$.
		Yellow is the optimal strategy $\assetsProcess$ for our mLOB-model with no additional impatience ($\gamma=0$); %
		blue is with impatience $\gamma > 0$ %
		such that our mLOB-model liquidates in time $T=1$;
		green is with `patience' ($\gamma < 0$, i.e.\ reduced downwards drift $\delta=\gamma-\mu$ ) such that $\assetsProcess_T = \EE[X_{T-}]$. %
	}
	\label{fig:ls-bbf-strategies-impatience}
\end{figure}
\Cref{fig:ls-bbf-strategies-impatience} displays common realizations for optimal strategies for the LS-model and (three variants of) the mLOB-model with parameters chosen to highlight their qualitative differences.
In the mLOB-model there is no terminal block trade due to the infinite time horizon.
Instead one can choose impatience $\gamma$ to adjust the time-to-liquidation (blue line) or -- in a more ad-hoc manner -- to reach a given amount of assets at some fixed time (green line).
The initial position $X_{0-} = \baseS_0$ to be liquidated is taken to be large, being the total amount of shares offered at positive prices on the bid side of the block-shaped additive order book (at $t=0$).
Hence, for the considered Black-Scholes model, the probability $p:= \PP\brackets{ \exists t \in [0,T]: S^{X^T}_t < 0 }$ of observing negative prices $S^X_t$ under the optimal strategy $X=X^T$ in the LS-model could be high if $T,\sigma,X_{0-}$ are not sufficiently small: for parameters as from \cref{fig:ls-bbf-strategies-impatience} one obtains e.g.\ $p \approx 0.7$. 

Although (unaffected) returns $d\baseS/\baseS$ are i.i.d.\ and the postulated order book shape is invariant over time, the figure shows frequent and stochastic fluctuations between buying and selling for the  LS-model. 
In this sense, the optimal strategy in the LS-model exhibits transaction-triggered price manipulation in the spirit of \cite[Definition~1]{AlfonsiSchiedSlynko12} (in continuous time) for negative drift $\mu<0$, whereas such is not the case in the mLOB-model for  $Y_0$  being zero by \cref{thm:optimal liquidation with intermediate buy orders}, cf.\ \cref{rem:round-trips-in-depressed-market}.
Let us note that in the Lorenz-Schied model one would obtain a deterministic optimal strategy (of bounded variation) if the unaffected base price would be taken to be not of (multiplicative) Black-Scholes but of (additive) Bachelier type $\!\diff S^0_t = \mu \diff t + \sigma \diff W_t$. 

This indicates that additive impact models are better suited for additive (Bachelier) price dynamics, while a multiplicative impact model suits multiplicative (Black-Scholes) price dynamics.
Let us remark that, of course, stochastic optimal strategies can also arrise in multiplicative models, if relevant state variables are stochastic, see e.g.\ \cite{BechererBilarevFrentrup2016-stochastic-resilience}.
It is fair to note that additive models for asset prices and  price impact
have the benefit of easier analysis, in particular for the martingale case without drift.
We believe that multiplicative models offer benefits from a conceptual point of view and also for applications where time horizon is not small.
Liquidation of an asset position that is very large (relative to LOB depth), say by an institutional investor, clearly could  require a longer horizon; 
the econometric studies \cite[p.1152]{ChanL95} or \cite{MaugEtal}, e.g., consider trade sequences of a month.  
Also optimal investment and hedging problems may be posed for maturities not being small, cf.\ \cite{PhamEtal07,BouchardLoeperZhou16}.
On short horizons, additive models may provide good tractable approximations for practical implementation, as probabilities for negative (model) prices can be small, see e.g.\ \cite{SchachermayerTeichmann, Forsyth12} or \cite[p.514]{AlfonsiSchiedSlynko12}.

\end{example}

\appendix
\section{Appendix}

\label{appendix:proofs about value functions}

To prove the variational inequalities that are essential for verification, it will help to have
\begin{lemma}\label{lemma:V on boundary as integral}
	For all $\theta \ge 0$ we have
	\[
		\Vboundary(\theta) = \int_0^\theta f(y(x)) \exp\paren[\bigg]{ \int_x^\theta \frac{\delta}{h(y(z))} \diff z } \exp\paren[\bigg]{ \int_{y(\theta)}^{y(x)} \frac{\delta}{h(y)} \diff y } \diff x.
	\]
\end{lemma}
\begin{proof}
	Using \cref{eq:theta' common}, one gets
	\begin{align*}
		\hspace{-\mathindent}\int_0^\theta \!\! \frac{\delta}{h(y(z))} \diff z &= \int_{y_0}^{y(\theta)} \frac{\delta}{h(y)} \theta'(y) \diff y
	\\		&= \int_{y_0}^{y(\theta)} \frac{\delta}{h(y)} \paren[\bigg]{1 + \frac{h(y) \lambda(y)}{\delta} - \frac{h(y)h''(y)}{\delta h'(y)} + \frac{h(y) \paren[\big]{h\lambda + h' + \delta}'(y)}{\delta \paren[\big]{h\lambda + h' + \delta}(y)}} \diff y
	\\		&= \squeeze[.5]{ \int_{y_0}^{y(\theta)} \!\! \frac{\delta}{h(y)} \diff y + \brackets[\big]{\log f(y)}_{y_0}^{y(\theta)} - \brackets[\big]{\log {h'(y)}}_{y_0}^{y(\theta)} + \brackets[\Big]{\log {\paren[\big]{h\lambda + h' + \delta}(y)}}_{y_0}^{y(\theta)} }
	.\end{align*}
	Thus it follows
	\[
		\exp\paren[\bigg]{ \int_0^\theta \frac{\delta}{h(y(z))} \diff z } = \frac{1}{f(y_0)} \exp\paren[\bigg]{ \int_{y_0}^{y(\theta)} \frac{\delta}{h(y)} \diff y }  \paren[\bigg]{ \frac{f\, (h\lambda + h' + \delta)}{h'} }\!(y(\theta))
	,\]
	which implies
	\[
		\squeeze[.1]{ \exp\paren[\bigg]{ \int_x^\theta \!\! \frac{\delta}{h(y(z))} \diff z + \!\int_{y(\theta)}^{y(x)} \!\! \frac{\delta}{h(y)} \diff y } = \paren[\bigg]{ \frac{f\, (h\lambda + h' + \delta)}{h'} }\!(y(\theta)) \paren[\bigg]{ \frac{h'}{f\, (h\lambda + h' + \delta)} }\!(y(x)) }
	.\]
	Integration using \cref{lemma:integral for Vboundary} after multiplication with $f(y(x))$ yields the claim.
\end{proof}

\begin{lemma}\label{lemma:integral for Vboundary}
	Let $\theta \ge 0$.
	Then
	\(
		\int_0^\theta \paren[\big]{ \frac{h'}{h\lambda + h' + \delta} }\!(y(x)) \diff x = \paren[\big]{ \frac{h\, (h\lambda + \delta)}{\delta (h\lambda + h' + \delta)} }\!(y(\theta))
	.\)
\end{lemma}
\begin{proof}
	At $\theta = 0$, both sides equal zero, so it suffices to show equality of their derivatives.
	By \cref{eq:theta' common}, we have as functions of $y = y(\theta)$:
	\begin{align*}
		\squeeze[2]{ \frac{h'}{h\lambda + h' + \delta}\theta' } &= \frac{h'}{h\lambda + h' + \delta} \paren[\bigg]{1 + \frac{h \lambda}{\delta} - \frac{h h''}{\delta h'} + \frac{h \, (h\lambda + h' + \delta)'}{\delta \, (h\lambda + h' + \delta)}}
	\\		&= \frac{h' \, (\delta + h\lambda) (h\lambda + h' + \delta) - h h'' \, (h\lambda + h' + \delta) + h h' \, (h\lambda + h' + \delta)' }{\delta \, (h\lambda + h' + \delta)^2}	
	\\		&= \frac{h' \, (h\lambda + h' + \delta)^2 - \paren[\big]{(h')^2 + h h''} (h\lambda + h' + \delta) + h h' \, (h\lambda + h' + \delta)'}{\delta \, (h\lambda + h' + \delta)^2}
		\hspace{-1em}
	\\		&= \frac{h'}{\delta} - \paren[\bigg]{ \frac{hh'}{\delta \, (h\lambda + h' + \delta)} }'
			= \paren[\bigg]{ \frac{h \, (h\lambda + \delta)}{\delta \, (h\lambda + h' + \delta)} }'. \qedhere
	\end{align*}
\end{proof}

\begin{lemma}\label{lemma: wait region inequality}
	We have inequality~\eqref{eq:var ineq IV}, i.e.\  $V^{\scW}_y + V^{\scW}_\theta > f$, holding in $\scW$.
\end{lemma}
\begin{proof}
	It suffices to prove that $y\mapsto \bp{(V^{\scW}_y + V^{\scW}_\theta)/f}(y, \theta)=:k(y)$ is strictly decreasing on $(-\infty, y(\theta)]$ by smooth pasting on the boundary (i.e.~$k(y(\theta)) = 1$), see \cref{subsection: smooth-pasting}. Using notation from \cref{subsection: smooth-pasting}, we have for $y \leq y(\theta)$
	\begin{align*}
		f(y) k'(y) = C(\theta)\bp{ \phi''(y) - \lambda(y)\phi'(y) } + C'(\theta) \bp{ \phi'(y) - \lambda(y)\phi(y) }. 
	\end{align*}
	Since $\phi' = -\delta \phi/h$, we obtain  $\phi' - \lambda\phi = -\phi(\lambda h + \delta)/h$ and $\phi'' - \lambda\phi' = \delta \phi (\lambda h + h' + \delta) /h^2$.
	They are both negative for $y< y_\infty$.
	Hence, $k' < 0$ on $(-\infty, y_\infty)$ because $C(\theta),C'(\theta)> 0$ for $\theta > 0$, see \eqref{eq:C on the boundary}.
	To verify $f(y) k'(y) < 0$ on $[y_\infty, y(\theta))$, using \eqref{eq:C on the boundary} to get the form of $C'/C$, we check that $y\mapsto \bp{\frac{\lambda h + h' + \delta}{h(\lambda h + \delta)}}(y)$ is strictly increasing on $[y_\infty, y_0)$.
	Indeed,
	\[
		\hspace{-0.5\mathindent}
		\squeeze[2]{
			h^2(\lambda h + \delta)^2\frac{\diff}{\diff y} \paren[\bigg]{ \frac{\lambda h + h' + \delta}{h(\lambda h + \delta)} }
			= \underbrace{(\lambda h + h' + \delta)'}_{>0}\underbrace{h (\lambda h + \delta)}_{\geq 0} - \underbrace{(\lambda h + h' + \delta)}_{\geq 0} \bp{ h(\lambda h + \delta) }'
			>0
		}
	\]
	on $[y_\infty, y_0)$, because $h$ and $\lambda h + \delta$ are increasing and negative there, and thus $h(\lambda h + \delta)$ is strictly decreasing, see \cref{cond:model parameters}.
\end{proof}

\noindent
Recall from \cref{subsection: Construction of V and the optimal strategy} the regions $\scS_1 :=\{(y,\theta)\in \RR \times [0,\infty) \mid y(\theta) < y < y_0 + \theta \}$ and $\scS_2 :=\{(y,\theta)\in \RR \times [0,\infty) \mid y_0 + \theta < y \}$.

\begin{lemma}\label{lemma: simple sell region inequality}
	We have inequality~\eqref{eq:var ineq II},  i.e.\ $-\delta V^{\scS_2} - h(y) V^{\scS_2}_y < 0$, holding in $\scS_2$.
\end{lemma}
\begin{proof}
	Fix ${y > y_0}$.
	We will see that $g(\theta):= \delta V^{\scS_2}(y,\theta) + h(y)V^{\scS_2}_y(y,\theta)$ increases for $\theta \in (0,y-y_0]$.
	By \cref{eq: V in simple sell region} we find
	\begin{align*}
		g'(\theta) &= \frac{\diff}{\diff \theta} \paren[\Big]{\delta \int_{y-\theta}^y f(x)\diff x + h(y) \paren[\big]{f(y) - f(y-\theta)}}
		\\		&= \delta f(y-\theta) + h(y) f'(y-\theta)
				= f(y-\theta) \paren[\big]{\delta + h(y)\lambda(y-\theta)}
		\\		&> f(y-\theta) \paren[\big]{\delta + h(y-\theta)\lambda(y-\theta)}
				\ge f(y-\theta) \paren[\big]{\delta + h(y_0)\lambda(y_0)} = 0\,,
	\end{align*}
	by monotonicity of $h$ and $h\lambda$.
	Noting $g(0) = 0$, the claimed inequality follows.
\end{proof}

\begin{lemma}\label{lemma: difficult sell region inequality}
	We have inequality~\eqref{eq:var ineq II}, $-\delta V^{\scS_1} - h(y) V^{\scS_1}_y < 0$, holding in $\scS_1$.
	Moreover
	\begin{align}
		\Vboundary'(\theta) &= f(y(\theta)) + \frac{\delta}{h(y(\theta))} \paren[\big]{1 - y'(\theta)} \Vboundary(\theta) \label{eq: theta derivative of Vboundary}
	\\\hspace{-\mathindent}\text{and }\qquad %
		V^{\scS_1}_y(y,\theta) &= f(y) - f(y-\Delta) - \frac{\delta}{h(y-\Delta)} \Vboundary(\theta - \Delta). \label{eq: y derivative of V in difficult sell region}
	\end{align}
\end{lemma}
\begin{proof}
	Let $(y,\theta) \in \widebar\scS_1$.
	By \eqref{def:distance to boundary}, we have $\theta - \Delta = \theta(y - \Delta)$, implying
	\begin{align*}
		 \Delta_y &= \frac{\theta'(y-\Delta)}{\theta'(y-\Delta) - 1} = \frac{1}{1 - y'(\theta - \Delta)}.
	\end{align*}
	Using \cref{lemma:V on boundary as integral}, we get $\Vboundary'(\theta) = f(y(\theta)) + \frac{\delta}{h(y(\theta))} \paren[\big]{1 - y'(\theta)} \Vboundary(\theta)$ and thereby
	\begin{align*}
		 \Vboundary'(\theta - \Delta) &= f(y-\Delta) + \frac{\delta}{h(y - \Delta)} \paren[\big]{1 - y'(\theta - \Delta)} \Vboundary(\theta - \Delta).
	\end{align*}
	With \cref{eq: unsimplified V in difficult sell region} it follows that
	\begin{align*}
		V^{\scS_1}_y(y,\theta) &= \Vboundary'(\theta-\Delta) \cdot (-\Delta_y) + f(y) - f(y-\Delta) (1-\Delta_y)
	\\		&= f(y) - f(y-\Delta) - \frac{\delta}{h(y-\Delta)} \Vboundary(\theta - \Delta).
	\end{align*}
	Now we  fix $(y_b, \theta_b):= (y - \Delta, \theta - \Delta)$ on the boundary and vary $\Delta \ge 0$ to show monotonicity of $g(\Delta):= \delta V^{\scS_1}(y_b + \Delta, \theta_b + \Delta) + h(y_b + \Delta) V^{\scS_1}_y(y_b + \Delta, \theta_b + \Delta)$, which equals
	\begin{align*}
		 \delta \Vboundary(\theta_b) \paren[\bigg]{1 - \frac{h(y_b+\Delta)}{h(y_b)}} + \delta \int_{y_b}^{y_b + \Delta}f(x) \diff x 
                  	+ h(y_b + \Delta) \paren[\big]{f(y_b + \Delta) - f(y_b)}
	\end{align*}
	and gives $g(0) = 0$.
	Therefore, one obtains
	\begin{align*}
		g'(\Delta) &= \delta \Vboundary(\theta_b)\frac{-h'(y_b+\Delta)}{h(y_b)} + \delta f(y_b + \Delta) 
	\\			&\qquad + h'(y_b + \Delta) \paren[\big]{f(y_b + \Delta) - f(y_b)} + h(y_b + \Delta) f'(y_b+\Delta)
	\\		&= \squeeze[2]{ -h'(y_b + \Delta) \paren[\bigg]{ \frac{\delta}{h(y_b)} \Vboundary(\theta_b) + f(y_b) } + f(y_b + \Delta) \, \paren[\big]{h\lambda + h' + \delta}\!(y_b + \Delta) }
	\\		&= -h'(y_b + \Delta) f(y_b) \frac{\paren[\big]{h\lambda + h' + \delta}\!(y_b)}{h'(y_b)} + f(y_b + \Delta) \, \paren[\big]{h\lambda + h' + \delta}\!(y_b + \Delta).
	\end{align*}
	Note that $k(y):= \paren[\big]{\paren{h\lambda + h' + \delta}/{h'}}(y)$ satisfies $k(y_\infty) = 0$, $k(y_0) = 1$, $k(y) > 1$ for $y > y_0$ and it is increasing in $(y_\infty,y_0)$ because we have on that interval
	\[
		k'(y)(h')^2
		> (h\lambda + h' + \delta)' h' - (h\lambda + h' + \delta)h'' - (h\lambda + \delta)(h\lambda)'
		= (h\lambda)'h' - (h\lambda + \delta)(h\lambda + h')' >\! 0.
	\]
	Hence, $f k$ is also increasing in $(y_\infty,y_0]$, which implies $g'(\Delta) > 0$ for $y_\infty < y_b \le y_b + \Delta \le y_0$.
	Now let $y_b + \Delta > y_0$.
	Here, we have $k(y_b+\Delta) > 1$ and therefore
	\[
		\frac{g'(\Delta)}{h'(y_b + \Delta)}
			= \paren[\big]{f k}(y_b + \Delta) - \paren[\big]{f k}(y_b)
			\ge \paren[\big]{f k}(y_b + \Delta) - \paren[\big]{f k}(y_0)
			> f(y_b + \Delta) - f(y_0) > 0.
	\]
	In conclusion, $g'(\Delta)>0$ for every $\Delta > 0$, which implies $g(\Delta) > 0$ for $\Delta > 0$.
\end{proof}

\begin{proof}[Proof of \cref{lemma: V continuously differentiable}]
	Inside $\scW$, $\scS_1$ and $\scS_2$, the function $V$ is already $C^1$ because of \cref{lemma: boundary,eq: V boundary}.
	The inequalities \labelcref{eq:var ineq II,eq:var ineq IV} are proven in \cref{lemma: wait region inequality,lemma: simple sell region inequality,lemma: difficult sell region inequality}, while \cref{eq:var ineq I,eq:var ineq III,eq:var ineq boundary cond} are clear by construction.
	
	Let $(y,\theta) \in \widebar\scW \cap \widebar\scS_1$, so $y = y(\theta)$ and $\Delta=0$.
	Continuity is guaranteed by construction.
	We have existence of the directional derivative $V^{\scW}_y + V^{\scW}_\theta$ by \cref{lemma: wait region inequality} and its proof also shows continuity at $y=y(\theta)$.
	It remains to show equality $V^{\scW}_y = V^{\scS_1}_y$ here.
	This is already done in the proof of \cref{lemma: difficult sell region inequality} as $g(0)=0$.

	Now, let $(y,\theta) \in \widebar\scS_1 \cap \widebar\scS_2$, i.e.~$y = y_0 + \theta$ and $\Delta = \theta$.
	Continuity follows from $\Vboundary(0) = 0$, since $h(y_0)\lambda(y_0) + \delta = 0$.
	By construction, the directional derivative $V_y + V_\theta$ exists in $\scS$, so it suffices to show equality of $V_y$ from the left and from the right.
	As shown in \cref{lemma: difficult sell region inequality}, we have 
	\[
		V^{\scS_1}_y(y_0 + \theta, \theta) = f(y_0 + \delta) - f(y_0) = V^{\scS_2}(y_0 + \theta, \theta)
	.\]
	Finally, let $(y,\theta) \in \RR \times \{0\}$.
	Since $h(y_0)\lambda(y_0)+\delta = 0$, it follows $V(\cdot,0)=0$ directly.
	We only need to show existence and continuity of $V_\theta(y,0) = \lim_{\theta \searrow 0} \frac{1}{\theta} V(y,\theta)$.
	Let $y < y_0$.
	As shown in \cref{lemma: difficult sell region inequality},
	\(
		\Vboundary'(\theta) = f(y(\theta)) + \delta \paren[\big]{1 - y'(\theta)} \Vboundary(\theta) / h(y(\theta)),
	\)
	which leads to
	\[
		V^{\scW}_\theta(y,\theta) = f(y(\theta))\exp\paren[\bigg]{\int_{y(\theta)}^y \frac{-\delta}{h(x)} \diff x} +  \frac{\delta}{h(y(\theta))} V^{\scW}(y,\theta)
	\]
	by definition~\eqref{eq:V wait} of $V^\scW$.
	By l'Hôpital's rule,
	\(
		\lim_{\theta \searrow 0} \frac{1}{\theta} V^\scW(y,\theta)
= \lim_{\theta \searrow 0} V^{\scW}_\theta(y,\theta)\) 
equals \( f(y_0)\exp\paren[\big]{\int_{y_0}^y \frac{-\delta}{h(x)} \diff x},
	\)
	which is continuous in $(-\infty, y_0]$ and equals $f(y_0)$ at $y=y_0$.
	For $y>y_0$ we get $V(y,\theta) = V^{\scS_2}(y,\theta)$, if $\theta \ge 0$ is small enough.
	Again by l'Hôpital,
	\(
		\lim_{\theta \searrow 0} \frac{1}{\theta}V^{\scS_2}(y,\theta) = \lim_{\theta \searrow 0} V^{\scS_2}_\theta(y,\theta) = f(y).
	\)
	Now let $y=y_0$.
	For all $\theta > 0$ we have $V(y_0,\theta) = V^{\scS_1}(y_0, \theta)$.
	By construction it is $V^{\scS_1}_\theta(y_0, \theta) = f(y_0) - V^{\scS_1}_y(y_0, \theta)$.
	So by \cref{eq: y derivative of V in difficult sell region},
  the limit \(\lim_{\theta \searrow 0} \frac{1}{\theta} V^{\scS_1}(y_0,\theta) = 
f(y_0) - \lim_{\theta \searrow 0} V^{\scS_1}_y(y_0, \theta)\) is equal to
	\[
		f(y_0) - \lim_{\theta \searrow 0} \paren[\Big]{f(y_0) - f(y_0-\Delta) - \frac{\delta}{h(y_0-\Delta)}\Vboundary(\theta - \Delta)}	= f(y_0),
	\]
	since $\Delta(y_0, \theta) \to 0$ for $\theta \to 0$ and  $h(y_0) \ne 0$.
\end{proof}

\begin{proof}[Proof of \cref{thm:optimal liquidation with intermediate buy orders}]
	That $\Vbuysell\in C^1(\RR\times [0,\infty))$ essentially follows from \cref{lemma: V continuously differentiable}.
	We  show that $\Vbuysell$ satisfies \eqref{eq:ineq when interm buying I} -- \eqref{eq:ineq when interm buying II}.
	It is clear by construction that \eqref{eq:ineq when interm buying I} holds true, so it remains to show \eqref{eq:ineq when interm buying II}.
	For $(y,\theta)\in \widebar{\scS}$ the inequality follows from \cref{lemma: simple sell region inequality,lemma: difficult sell region inequality}; note that we have equality only when $(y,\theta)$ is on the boundary between $\scS$ and $\buyReg$, or $\theta = 0$.
	Now suppose that $(y,\theta)\in \buyReg$.
	For simplicity of the exposition let $\widetilde{\Delta}(y,\theta) = -\Delta(y,\theta) \geq 0$ be the distance from $(y,\theta)$ to the boundary in direction $(1,1)$.
	We shall omit the arguments of $\widetilde{\Delta}$ to ease notation.
	Set $(y_b, \theta_b) := (y+\widetilde{\Delta}, \theta+\widetilde{\Delta})$.
	Then
	\begin{align*}
		\Vbuysell(y,\theta)& = \Vbuysell(y_b, \theta_b) - \int_y^{y_b} f(x)\diff x \qquad \text{and moreover}
	\\	\Vbuysell_y(y,\theta)& = \frac{\diff}{\diff y}\paren[\big]{ \Vbuysell(y + \widetilde{\Delta}, \theta + \widetilde{\Delta}) - \int_y^{y + \widetilde{\Delta}} f(x)\diff x } 
	\\		& = \squeeze[2]{(\widetilde{\Delta}_y+1)V_y + \widetilde{\Delta}_y V_\theta - \paren[\Big]{(1+\widetilde{\Delta}_y)f(y+\widetilde{\Delta}) - f(y)} 
			= f(y) - \Vbuysell_\theta(y_b, \theta_b)}\,,
	\end{align*}
	where the last equality uses $f = V_y + V_\theta$.
	We set
	\begin{align*}
		g(\widetilde{\Delta}):= -h(y_b - \widetilde{\Delta}) \paren[\Big]{f(y_b-\widetilde{\Delta}) - \Vbuysell_\theta(y_b, \theta_b)} - \delta \paren[\Big]{ \Vbuysell(y_b, \theta_b) - \!\int_{y_b - \widetilde{\Delta}}^{y_b} f(x)\diff x }.
	\end{align*}
	Note that $g(0) = 0$ by construction of the boundary between $\scS$ and $\buyReg$ in \cref{subsection: smooth-pasting}.
	Thus, it suffices to verify $g' \leq 0$.
	We have
	\(
		g'(\widetilde{\Delta}) = f(y)\paren[\big]{ h(y)\lambda(y) + h'(y) + \delta } - h'(y)\Vbuysell_{\theta}(y_b, \theta_b),
	\)
	recalling $y = y_b - \widetilde{\Delta}$.
	Recall the following form for $\Vbuysell$ on the boundary (see \eqref{eq:C on the boundary}):
	\begin{align}
		\Vbuysell_\theta(y_b, \theta_b)&= f(y_b) \frac{h(y_b)\lambda(y_b) + h'(y_b) + \delta}{h'(y_b)}. \label{eq:V_theta on the boundary}
	\end{align}
	Thus, checking that $g'(\widetilde{\Delta}) \leq 0$ is equivalent to verifying
	\begin{equation}\label{eq:ineq verif}
		h(y)\lambda(y) + h'(y) + \delta - \frac{h'(y)}{h'(y_b)} \cdot \frac{f(y_b)}{f(y)}\cdot \paren[\big]{ h(y_b)\lambda(y_b) + h'(y_b) + \delta } \leq 0.
	\end{equation}
	Since $y \leq y_b$ we have that $f(y)\leq f(y_b)$.
	Hence it suffices to check the last inequality when $f(y_b)/f(y)$ is replaced by 1.
	This is equivalent to verifying that $k(y) \le k(y_b)$ for $k:= (h\lambda + h' + \delta)/h'$. 
	For $y \le y_\infty$ this is trivial because $k(y) \le 0 < k(y_b)$ in that case.
	For $y_\infty < y \le y_b$ we can use monotonicity of $k$, which was shown in the proof of \cref{lemma: difficult sell region inequality}.

	Note that the analysis above actually shows that equality in \eqref{eq:ineq when interm buying II} holds if and only if $(y,\theta)$ is on the boundary between $\scS$ and $\buyReg$.
	This ensures uniqueness of the optimal strategy.
	The rest of the proof follows on the same lines as in the one for \cref{thm: optimal strategy}.
\end{proof}


\end{document}